\documentclass[10pt,reqno,b5paper]{amsproc}
\usepackage{anysize}
\marginsize{20mm}{20mm}{15mm}{15mm}
\usepackage[noadjust]{cite}
\usepackage[T1]{fontenc}
\usepackage{txfonts,amsmath,esint}
\usepackage[colorlinks,urlcolor=blue,citecolor=blue,linkcolor=blue]{hyperref}
\usepackage{amsmath}
\usepackage{amsfonts}
\usepackage{amssymb}
\usepackage{graphicx}
\usepackage[all]{xy}
\usepackage{arydshln}
\usepackage[english]{babel}

\DeclareSymbolFont{CM}{OMX}{cmex}{m}{n}
\DeclareMathSymbol{\sumop}{\mathop}{CM}{"50}
\renewcommand{\sum}{\sumop}

\newtheorem{thm}{\sc Theorem}[section]
\newtheorem{cor}[thm]{\sc Corollary}
\newtheorem{exmp}[thm]{\sc Example}
\newtheorem{lem}[thm]{\sc Lemma}
\newtheorem{prop}[thm]{\sc Proposition}
\theoremstyle{definition}
\newtheorem{defn}[thm]{\sc Definition}
\theoremstyle{remark}
\newtheorem{rem}[thm]{\it Remark}
\newtheorem{ques}[thm]{\it Question}

\newtheorem{conj}[thm]{\sc Conjecture}
\numberwithin{equation}{section}
\allowdisplaybreaks
\AtEndDocument{\bibliographystyle{siam}
\bibliography{RefBase}\ }
\begin{document}

\title[Existence and non-uniqueness of cone spherical metrics]{Existence and non-uniqueness of cone spherical metrics with prescribed singularities on a compact Riemann surface with positive genus}

\author{Yu Feng }
\address{Chern Institute of Mathematics and LPMC, Nankai University \newline \indent Tianjin 300071, China,}
\email{yuf@nankai.edu.cn}
\thanks{}

\author{Jijian Song}
\address{School of Mathematics and KL-AAGDM, Tianjin University\newline \indent Tianjin, 300350, China}
\email{smath@tju.edu.cn}
\thanks{}

\author{Bin Xu}
\address{CAS Wu Wen-Tsun Key Laboratory of Mathematics and  School of Mathematical \newline \indent Sciences, University of Science and Technology of China, Hefei 230026 China}
\email{bxu@ustc.edu.cn}
\subjclass[2020]{Primary 30F45; Secondary 14H60}

\keywords{cone spherical metric, ramification divisor map, polystable extension}


\vspace{-5mm}

\begin{abstract}
Cone spherical metrics, defined on compact Riemann surfaces, are conformal metrics with constant curvature one and finitely many cone singularities. Such a metric is termed \textit{reducible} if a developing map of the metric has monodromy in ${\rm U(1)}$, and \textit{irreducible} otherwise. Utilizing the polystable extensions of two line bundles on a compact Riemann surface $X$ with genus $g_X>0$, we establish the following three primary results concerning these metrics with cone angles in $2\pi{\mathbb Z}_{>1}$:

\begin{itemize}
\item[(1)]  Given an effective divisor $D$ with an odd degree surpassing $2g_X$ on $X$, we find the existence of an effective divisor $D'$ in the complete linear system $|D|$ that can be represented by at least two distinct irreducible cone spherical metrics on $X$.

\item[(2)] For a generic effective divisor $D$ with an even degree and $\deg D\geq 6g_X-2$ on $X$, we can identify an arcwise connected Borel subset in $|D|$ that demonstrates a Hausdorff dimension of no less than $\big(\deg D-4g_{X}+2\big)$. Within this subset, each divisor $D'$ can be distinctly represented by a family of reducible metrics, defined by a single real parameter.

\item[(3)] For an effective divisor $D$ with $\deg D=2$ on an elliptic curve, we can identify a Borel subset
in $|D|$ that is arcwise connected and exhibits a Hausdorff dimension greater than one. Within this subset, each divisor $D'$ can be distinctly represented by a family of reducible metrics, defined by a single real parameter.

\end{itemize}
\end{abstract}
\maketitle

\section{Introduction}
 We investigate the existence and uniqueness problem of cone spherical metrics representing effective divisors on a compact Riemann surface $X$ of genus $g_X>0$. This project builds upon the work of \cite{LSX:2021}, wherein the authors establish a correspondence between cone spherical metrics representing effective divisors and polystable extensions of two line bundles 
 on compact Riemann surfaces of genus $\geq 0$. This section will involve a review of several concepts and research backgrounds related to cone spherical metrics, along with an elucidation of the main results of this article.

\subsection{Background of cone spherical metrics}
Let $X$ be a compact Riemann surface of genus $g_X$, and $D = \sum_{j=1}^n \, \beta_j \, p_j$ an $\mathbb{R}$-divisor on $X$ such that $ -1<\beta_j \neq 0$. We call ${\rm d}s^2$ a {\it cone spherical metric representing} $D$ if

\begin{itemize}
\item ${\rm d}s^2$  is a conformal metric of Gaussian curvature $+1$ on $X \setminus {\rm supp} \, D := X \setminus \{p_1, \ldots, p_n\}$.

\item  ${\rm d}s^2$ has a cone singularity at each $p_j\in {\rm supp}\, D$ with cone angle $2 \pi( 1 + \beta_j)$. That is, there exists a complex coordinate chart $(U, \, z)$ centered at $p_j$ such that ${\rm d}s^2$ has form $e^{2 \varphi } \, |{\rm d}z|^2$  on $U\setminus \{ p_j \}$, where the real valued function $\big(\varphi - \beta_j \, \ln \, |z|\big)$ extends to a continuous function on $U$.
\end{itemize}
By the Gauss-Bonnet formula, if ${\rm d}s^2$ is a cone spherical metric representing $D$ on $X$, then there holds $2-2g_{X}+\deg \, D >0$,
where $\deg \, D = \sum_{j=1}^n \, \beta_j$. 
The study of the existence and uniqueness of cone spherical metrics has long been a pivotal, yet unresolved, question in classical mathematics, attracting considerable scholarly attention.
Troyanov \cite[\S3, Theorem 4]{Troyanov:1991} posited that such a metric exists under the condition:
\[
0<2-2g_{X}+\deg \, D <{\rm min}\left\{2,2\ {\rm min}\ \beta_j\right\}.
\]
Luo--Tian \cite{LT:1992} demonstrated that the aforementioned condition is also necessary and that the cone spherical metric is unique, provided that there are at least $n\geq 3$ cone angles, each falling within $(0,2\pi)$ and positioned at any $n$ distinct points on the Riemann sphere ${\mathbb P}^1$.
Subsequently, the synergistic interplay of Troyanov's finding with that of Luo-Tian culminated in a landmark theorem. It articulates the necessary and sufficient criteria for the existence and uniqueness of spherical metrics on ${\mathbb P}^1$ with $n\geq 3$ prescribed cone angles $\displaystyle{\{2\pi(1+\beta_j)\}_{j=1}^n\subset (0,\, 2\pi)}$, which are assigned to 
any $n$ distinct cone singularities on $\mathbb{P}^1$, respectively. This criteria is also called the stability condition with form
\[\sum_{\substack{j=1 \\ j\not=i}}^n\, \beta_j<\beta_i\quad \text{for all}\quad 1\leq i\leq n.\]
Notably, the criteria by Troyanov and Luo-Tian is not at all sensitive to the configuration of positions of cone 
singularities on the underlying Riemann surface.  De Borbon-Panov \cite{dBP:2022} provided a reproof of this theorem by utilizing the Kobayashi-Hitchin correspondence for rank two parabolic stable bundles on the Riemann sphere. Independently, the authors of \cite[Theorem 2.1.]{LSX:2021} observed that a developing map of an irreducible spherical metric could be thought of as an embedded line sub-bundle in some rank two stable vector bundle.   Other relevant studies include those by  Bartolucci-De Marchis--Malchiodi \cite{BdeMM:2011}, Chai-Lin-Wang \cite{Chai-L-W:2015}, Chen-Lin \cite{Chen-L:2015},  Chen-Lin-Wang \cite{Chen-L-W:2004}, Chen-Kuo-Lin \cite{Chen-K-L:2017}, Eremenko \cite{Eremenko:2004,  Eremenko:2020coaxial, Eremenko:2020four}, Eremenko-Gabrielov \cite{Eremenko-G:2015}, Eremenko-Gabrielov-Tasarov \cite{Eremenko-G-T:2014, Eremenko-G-T:2016four, Eremenko-G-T:2016three}, Lin-Wang \cite{Lin-W:2010, Lin-W:2017}, Mazzeo-Zhu \cite{MZ:2020, MZ:2022}, Mondello--Panov \cite{MP:2015, MP:2019}, Troyanov \cite{Troyanov:1989}, Umehara-Yamada \cite{Umehara-Y:2000}, Xu-Zhu \cite{XZ:2021} and Zhu \cite{Zhu:2020}.

In this manuscript, our primary emphasis is on exploring the existence and {\it non-uniqueness} of cone spherical metrics representing specific effective divisors on compact Riemann surfaces. We systematically divide these metrics into two distinct categories: reducible and irreducible, each warranting a separate, detailed examination. In the following sections, we will revisit and elucidate the pertinent concepts associated with these classifications.

Let ${\rm d}s^2$ be a cone spherical metric representing $D$ on $X$. This metric gives rise to a multivalued locally univalent holomorphic map $f\colon X \setminus {\rm supp} \, D \rightarrow \mathbb{P}^1$, referred to as the {\it developing map} of ${\rm d}s^2$.  The map $f$ satisfies the following properties (\cite[Lemma 2.1 and Lemma 3.1]{CWWX:2015}):
\begin{itemize}
\item [(1)] $\mathrm{d} s^2=f^{*}\mathrm{d} s_{0}^2$, where\ $\mathrm{d} s_{0}^2=\frac{4|{\rm d}z|^2}{(1+|z|^2)^2}$\ is the Fubini-Study metric on $\mathbb{P}^1$;

\item [(2)] The monodromy of $f$ lies in
\[ {\rm PSU(2)}:=\left\{ z \mapsto \frac{az+b}{-\overline{b}z + \overline{a}} \; \Big| \; |a|^2 + |b|^2 = 1 \right\} \subset {\rm PSL}(2,\,\mathbb{C});
 \]
\item [(3)] The principal singular term of the Schwarzian derivative $\{f,z\}$ of $f$ assumes the form $\frac{1-\beta_j^2}{2z^2}$
in some complex coordinate chart $(U,\, z)$ centered at each $p_{j}$.
\end{itemize}
Moreover, any two developing maps of $\mathrm{d} s^2$\ are interrelated via a M{\" o}bius transformation in ${\rm PSU}(2)$.
A cone spherical metric  is called {\it reducible} if and only if some developing map of the metric has monodromy in ${\rm U}(1) = \left\{ z \mapsto e^{\sqrt{-1}\, t}z \mid t \in [0, \, 2\pi) \right\}$. Otherwise, it is called {\it irreducible}.  We call a reducible metric {\it (non)trivial} if the monodromy of a developing map of the metric is (non)trivial.

In \cite{LSX:2021}, the authors established an algebraic framework for cone spherical metrics representing effective divisors. This approach, intricately connected to line sub-bundles of rank two polystable vector bundles, is expounded in \cite[Theorem 3.1.]{LSX:2021} and is fundamentally reliant on the concept of developing maps of cone spherical metrics. 
Considering that the foundation of our current study is this established framework, we will dedicate the remaining paragraphs of this subsection to a comprehensive overview of its essential principles.

Let ${\rm d}s^2$ be a cone spherical metric representing an effective divisor $D$. Consider a coordinate covering $\{U_{\alpha},z_{\alpha}\}$ of $X$, where each $U_{\alpha}$ is sufficiently small and contains at most one point in $\textup{Supp}\  D$. On each $U_{\alpha}$, there exists a holomorphic map $f_{\alpha}\colon U_{\alpha}\rightarrow \mathbb{P}^1$ such that ${\rm d}s^2|_{U_{\alpha}}=f_{\alpha}^{*}\mathrm{d} s_{0}^2$. On $U_{\alpha}\cap U_{\beta}$, there holds $f_{\alpha}=f_{\alpha\beta}\circ f_{\beta}$, where $f_{\alpha\beta}\in \textup{PSU(2)}$. Consequently,
$\{U_{\alpha},f_{\alpha}\}$ defines a branched projective covering of $X$ with ramification divisor $D$. The transition functions $\{f_{\alpha\beta}\}$ define an indigenous bundle $P$, while $\{f_{\alpha}\}$ provide a holomorphic section $s$ of $P$ that is 
{\it non-locally flat}, i.e., $f_{\alpha}$ is non-constant on each $U_{\alpha}$ (\cite[Definition 3.1.]{LSX:2021}).

In contrast, given a projective unitary flat $\mathbb{P}^1$ bundle $P$ coupled with a non-locally flat section $s$, denoted as the pair $(P,s)$, it is possible to construct a cone spherical metric. A projective unitary flat $\mathbb{P}^1$ bundle is characterized by having a bundle atlas whose transition functions are constants and elements of ${\rm PSU(2)}$. Suppose $\{U_{\alpha},s_{\alpha}\}$ represents the local expression of the section $s$ to such a bundle atlas. Then, on $U_{\alpha}\cap U_{\beta}$, there holds $s_{\alpha}=f_{\alpha\beta}\cdot s_{\beta}$, where the transition functions $f_{\alpha\beta}\in {\rm PSU(2)}$. Let ${\mathrm d}s^{2}_{\alpha}$ denote the singular metric defined by $s^{*}_{\alpha}(\mathrm{d} s_{0}^2)$ on $U_{\alpha}$. Since the transition functions $f_{\alpha\beta}$ are in ${\rm PSU(2)}$, this construction ensures the definition of a globally consistent cone spherical metric, symbolized by ${\mathrm d}s^{2}$,  which represents an effective divisor on $X$, i.e., the ramification divisor of $s$ concerning the unitary flat bundle $P$.

In summary, the authors \cite[Theorem 3.1.]{LSX:2021}  established  the existence of a surjective map on compact Riemann surfaces. This map originates from the set of projective unitary flat $\mathbb{P}^1$ bundles, denoted as $P$, each coupled with a non-locally flat section $s$, and leads to the set of cone spherical metrics representing effective divisors. It is noteworthy that on compact Riemann surfaces, the pair $(P, s)$ corresponds to an embedding $L\stackrel{i} \hookrightarrow E$ of the line bundle $L$ into the rank two polystable vector bundle $E$, with $P = \mathbb{P}(E)$, where $L$ is the pre-image of $s$. In this context, $E$ signifies a polystable extension of $E/L$ by $L$. Let $\mathcal{PE}(X)$ denote the space of polystable extensions of two line bundles on $X$. The space of cone spherical metrics representing effective divisors on $X$ is denoted by $\mathcal{M}(X, \mathbb{Z})$. Consequently, the aforementioned surjective map induces another surjective map, as shown in the following equation:
\begin{equation}
\label{equ:surmap}
\sigma \colon \mathcal{PE}(X) \to \mathcal{M}(X, \mathbb{Z}).
\end{equation}
We will frequently utilize it throughout the proof of our main results.

\subsection{Main results}
Our primary results are segmented into two parts. The first entails a theorem on the existence of non-unique irreducible metrics representing effective divisors of odd degrees. The theorem is stated as follows.

\begin{thm}
\label{thm:irr}
If $D$ is an effective divisor of odd degree greater than $2g_X$ on a compact Riemann surface $X$ of genus $g_X>0$ . Then there exists another effective divisor $D'$ in $|D|$ and two distinct irreducible cone spherical metrics on $X$ which represent $D'$.
\end{thm}

For reducible metrics, depending on whether $g_X\geq 2$ or $g_X=1$, we have two distinct results concerning their existence as well as non-uniqueness:

\begin{thm}
\label{thm:gen}
Let $X$ be a compact Riemann surface of genus $g_X\geq 2$, and let $D$ be an effective divisor with even degree 
$\deg\, D\geq 6g_X-2$ on $X$. Then, there exists an effective divisor $D'\in |D|$ that can be represented by a reducible metric. Furthermore, the following two alternatives hold:

\begin{itemize}
\item If $D'$ is represented by a nontrivial reducible metric, there exists an arcwise connected Borel subset of Hausdorff dimension $\geq\deg\, D-4g_{X}+2$ in $|D|$ containing $D'$. This subset includes divisors that can be represented by nontrivial reducible metrics. Moreover, every divisor in this Borel set can be represented by a family of reducible metrics defined by one real parameter.
\item If $D'$ is represented by a trivial reducible metric, there exists an arcwise connected Borel subset of Hausdorff dimension $\geq\deg\, D-4g_{X}$ in $|D|$ containing $D'$. This subset includes divisors that can be represented by trivial reducible metrics. Furthermore, every divisor in this Borel set can be represented by a family of reducible metrics defined by three real parameters.
\end{itemize}
\end{thm}

A similar theorem regarding reducible metrics on elliptic curves is stated below.

\begin{thm}
\label{thm:ell}
Let $X$ be an elliptic curve and $D$ an effective divisor of
${\rm even}=\deg\, D\geq 2$ on $X$.
Then there exists an effective divisor $D'\in |D|$ which is represented by a family of reducible metrics, defined by one or three real parameters.
Moreover, if $\deg\, D\geq 4$, then the following alternatives hold:
\begin{itemize}
\item If $D'$ is represented by a nontrivial reducible metric, there exists an arcwise connected Borel subset of Hausdorff dimension $\geq\deg\, D-2$ in $|D|$ containing $D'$. This subset includes divisors that can be represented by nontrivial reducible metrics. Moreover,  every divisor in this Borel set can be represented by a family of nontrivial reducible metrics defined by one real parameter.
\item If $D'$ is represented by a trivial reducible metric, there exists an arcwise connected Borel subset of Hausdorff dimension $\geq\deg\, D-4$ in $|D|$ containing $D'$. This subset includes divisors that can be represented by trivial reducible metrics. Furthermore, every divisor in this Borel set can be represented by a family of trivial reducible metrics defined by three real parameters.
\end{itemize}
If $\deg\,D=2$, then $D'$ must be represented by a nontrivial reducible metric, and there exists an arcwise connected Borel subset of Hausdorff dimension $\geq 1$ in $|D|$ containing $D'$. Moreover,  every divisor in this Borel set can be represented by a family of nontrivial reducible metrics defined by one real parameter.
\end{thm}

We present a corollary of the preceding two theorems, as stated below.

\begin{cor}
\label{cor:generic}
Let $X$ be a compact Riemann surface of genus $g_X \geq 1$, and consider $D$ as a generic effective divisor on $X$ with an even degree, satisfying $\deg D \geq 6g_X - 2$.
Under these conditions, we can identify an arcwise connected Borel subset within the linear system $|D|$, which exhibits a Hausdorff dimension of at least $\deg D - 4g_X+2$. Within this subset, each divisor $D'$ can be uniquely associated with a family of nontrivial reducible metrics, characterized by a single real parameter.
\end{cor}

\subsection{Outline}
Section 2 provides a detailed explanation of the ramification divisor map, employing the framework of polystable rank two vector bundles and their line subbundles. This map serves to describe the divisors associated with cone spherical metrics. Additionally, the fundamental lemma for reducible metrics in our story is proven, laying the foundation for the further research. In Section 3, by applying Brouwer's invariance of domain theorem to the ramification divisor map, we prove Theorem \ref{thm:irr}. This demonstrates a kind of non-uniqueness of irreducible metrics representing effective divisors of odd degree. Section 4 serves as a theoretical foundation for the investigation of the existence and uniqueness of reducible metrics. In Section 5, we define the restricted ramification divisor map in Proposition \ref{prop:restricted-rdm} and compute the complex differential of this map. By using this map, we obtain Theorems \ref{thm:gen}, \ref{thm:ell}, and Corollary \ref{cor:generic}, which address the existence and non-uniqueness results of reducible metrics. Finally, in Section 6, we introduce two open questions that have the potential to enhance our conclusion and stimulate further discussion.

\section{Ramification divisor map}

In this section, we present a comprehensive exposition of the {\it ramification divisor map} within the context of polystable rank two vector bundles and their line subbundles. This map plays a crucial role in characterizing the divisors associated with cone spherical metrics. Furthermore, we establish the fundamental lemma for reducible metrics, which serves as a cornerstone for subsequent investigations in our study of reducible metrics.

\subsection{Revisiting the definition of the ramification divisor map}
In \cite[Theorem 2.1]{LSX:2021}, the authors introduced the ramification divisor map to investigate irreducible metrics. Due to its deep connection with cone spherical metrics, this subsection aims to provide a comprehensive description of the map (refer to Proposition \ref{prop:rdm}) in a unified manner for both reducible and irreducible metrics.

Let $X$ be a compact Riemann surface endowed with a K\"{a}hler form $\omega$. We consider the following polystable extension:
\begin{equation}
\label{equ:strictly poly-stable}
\mathbb{E}:0\to L\stackrel{\iota}{\to} E\stackrel{p}{\to} M\to 0,
\end{equation}
where $E$ is a polystable rank $2$ vector bundle on $X$.
Subsequently, we can choose a Hermitian-Einstein metric $h_{E}$ on $E$.  By utilizing $\mathbb{E}$ and $h_{E}$, we can equip
$L$ with the induced metric, while $M$ is endowed with the quotient metric.
Indeed, the Hermitian metric $h_{E}$ on $E$ induces a $\mathcal{C}^{\infty}$ splitting $p^{*}\colon M\to E$ such that $p\circ p^{*}={\rm Id}_{M}$. The adjoint homomorphism $\iota^{*}\colon E\rightarrow L$ serves as the orthogonal projection of $E$ onto $L$.
We can define $h_{L}(\cdot,\cdot):=h_{E}(\iota\cdot,\iota\cdot)$ and $h_{M}(\cdot,\cdot):=h_{E}(p^{*}\cdot,p^{*}\cdot)$. Considering the corresponding Chern connections $D_{E}$, $D_{L}$, $D_{M}$, and utilizing the $\mathcal{C}^{\infty}$ isomorphism $\iota^{*}\oplus p\colon E\to L\oplus M$, we can express $D_{E}$ as (\cite[p.273]{Demailly:2012}):
\begin{equation}
\label{equ:matrix}
D_{E} = \begin{pmatrix}
D_{L} & -\beta \\
\beta^{*} & D_{M}
\end{pmatrix},
\end{equation}
where $\beta^{*}$ is called the second fundamental form of $L$ in $E$.
It is worth noting that $h_{L}(s,\beta t)=h_{M}(\beta^{*}s,t)$ for any $s\in \Gamma(L)$ and $t\in\Gamma(M)$. Since $D_{E}$ is Hermitian-Einstein, we can observe the following curvature equation:
\begin{equation*}
\label{equ:curvature}
\begin{split}
\lambda\omega \otimes {\rm Id}_{E} = D_{E} \circ D_{E} &= \begin{pmatrix}
D_{L} \circ D_{L} - \beta \wedge \beta^{*} & - \left(D_{L} \circ \beta + \beta \circ D_{M}\right) \\
\beta^{*} \circ D_{L} + D_{M} \circ \beta^{*} & -\beta^{*} \wedge \beta + D_{M} \circ D_{M}
\end{pmatrix} \\
&= \begin{pmatrix}
\Theta_{L} - \beta \wedge \beta^{*} & -D_{{\rm Hom}(M, L)} (\beta)\\
D_{{\rm Hom}(L, M)}(\beta^{*}) & - \beta^{*} \wedge \beta + \Theta_{M}
\end{pmatrix},
\end{split}
\end{equation*}
where $\lambda=-2\pi \sqrt{-1}\frac{\mu{(E)}}{\rm{vol}(X,\,\omega)}$ and $\mu{(E)}=\frac{\deg E}{{\rm rk} E}$ is the slope of the bundle $E$.

Thus, we obtain $D_{{\rm Hom}(L, M)}(\beta^{*}) = 0$, which implies $\bar{\partial}_{D_{{\rm Hom}(L, M)}}(\beta^{*}) = 0$ due to $X$ being a Riemann surface. Therefore, we conclude that $\beta^{*}$ is holomorphic, i.e. $\beta^*\in H_{\bar{\partial}}^{1,0}\big(X, {\rm Hom}(L, M)\big)$, where $H_{\bar{\partial}}^{1,0}\big(X, {\rm Hom}(L, M)\big)$ denotes the space of holomorphic 1-forms valued in ${\rm Hom}(L, M)$. Alternatively, $\beta^{*}$ can be described as a holomorphic section of the line bundle ${\rm Hom}(L, M)\otimes K_{X}$, with $K_{X}$ representing the canonical line bundle on $X$.

Furthermore, using the expression \eqref{equ:matrix}, we derive that 
\begin{equation}
\label{equ:2ndfund}
\beta^{*}=(p\otimes {\rm id})\circ D_{E}\circ \iota=(p\otimes {\rm id})\circ \partial_{E}\circ \iota,    
\end{equation}
where $\partial_{E}$ denotes the $(1,0)$-part of $D_{E}$. Specifically, if $s$ is a local section of $L$, then $\iota(s)$ is a local section of $E$. Through the $\mathcal{C}^{\infty}$ isomorphism $\iota^{*}\oplus p\colon E\to L\oplus M$, we deduce that $\iota^{*}\circ\iota(s)$ is a local section of $L$, and $p\circ \iota(s)$ is a a local section of $M$. Thus, we obtain:
\begin{equation*}
D_{E}\circ \iota(s)= \begin{pmatrix}
D_{L} & -\beta \\
\beta^{*} & D_{M}
\end{pmatrix} \begin{pmatrix}
\iota^{*}\circ\iota(s) \\
p\circ \iota(s)
\end{pmatrix}= \begin{pmatrix}
D_{L} & -\beta \\
\beta^{*} & D_{M}
\end{pmatrix} \begin{pmatrix}
s \\
0
\end{pmatrix}=\begin{pmatrix}
D_{L}(s) \\
\beta^{*}(s)
\end{pmatrix},
\end{equation*}
Thus $(p\otimes {\rm id})\circ D_{E}\circ \iota(s)=\beta^{*}(s)$.

\begin{lem}
\label{lem:well_defn_1}
If we choose another Hermitian-Einstein metric $\tilde{h}_{E}$ on $E$, we will still obtain the same $(1,0)$-form $\beta^{*}$.
\end{lem}
\begin{proof}
If $E$ is a stable vector bundle, then $\tilde{h}_{E}=\lambda h_{E}$, where $\lambda$ is a positive constant. In this case, the pair $(\bar{\partial}_{E},\tilde{h}_{E})$ yields the same Chern connection $D_{E}$ as the pair $(\bar{\partial}_{E},h_{E})$.

Assuming that $E$ is a polystable bundle but not stable, we can conclude that $E$ has an even degree. By tensoring with a line bundle, we can consider the following exact sequence:
\[
0\to L\stackrel{\iota}{\to} E\stackrel{p}{\to} L^{-1}\to 0.
\]
Since $\det E\cong \det L\otimes \det L^{-1}= \mathcal{O}_{X}$ and $\deg E=0$, we have $E=J \oplus J^{-1}$, where $J\in {\rm Pic}^0(X)$. Let $h$ be a Hermitian-Einstein metric on $J$. Then $h$ gives rise to a $\mathbb{C}$-antilinear isomorphism $$\varphi \colon J \to J^{-1},\quad \varphi(u)(v):=h(v,u)$$
as complex line bundles, and hence induces a metric $h^{*}$ on $J^{-1}$ with the form
$$h^{*}\big(\varphi(u),\varphi(v)\big):=h(v,u).$$
It can be verified that $h^{*}$ is a Hermitian metric on $J^{-1}$, and that it is also Hermitian-Einstein. Since $\deg J=\deg J^{-1}=0$, the metric $h_{E}=h\oplus h^{*}$ is Hermitian-Einstein on $J\oplus J^{-1}$.

If one opts for an alternative Hermitian-Einstein metric $\tilde{h}_{E}$ on $J\oplus J^{-1}$,  it can be expressed as  $\tilde{h}_{E}=a h\oplus b h^{*}$, where $a$ and $b$ are positive constants. It can be verified that the same Chern connection arises from the pairs $(\bar{\partial}_{E},h_{E})$ and $(\bar{\partial}_{E},\tilde{h}_{E})$. Actually, given that $E=J\oplus J^{-1}$, we have $h_{E}=h\oplus h^{*}$ and $\bar{\partial}_{E}=\bar{\partial}_{J}\oplus \bar{\partial}_{J^*}$.  Notably, $D_{J}$ is the Chern connection for both  $(\bar{\partial}_{J},h)$ and $(\bar{\partial}_{J},ah)$, and $D_{J^*}$ is the Chern connection for both $(\bar{\partial}_{J^*},h^*)$ and $(\bar{\partial}_{J^*},bh^*)$, where $a$ and $b$ are positive constants. Therefore, $D_{E}=D_{J}\oplus D_{J^*}$ is the Chern connection for both $(\bar{\partial}_{E},h_{E})$ and $(\bar{\partial}_{E},\tilde{h}_{E})$ pairs. Using the formula $\beta^{*}=(p\otimes {\rm id})\circ \partial_{E}\circ \iota$, where $\partial_{E}$ denotes the $(1,0)$-part of $D_{E}$, we obtain the same $\beta^{*}$ in both cases.
\end{proof}

Suppose that there exists another short exact sequence
$0 \to L \to F \to M \to 0$
which is equivalent to
(\ref{equ:strictly poly-stable}) in ${\rm Ext}_{X}^1(M,\, L)$, i.e. there exists a commutative diagram of holomorphic bundle maps as follows:
\begin{displaymath}
\xymatrix{
0 \ar[r] & L \ar@{=}[d] \ar[r] & E\ar[d]\ar@{->}^{f}[d]  \ar[r] & M \ar@{=}[d] \ar[r] &0 \\
0 \ar[r] & L \ar[r]                        & F \ar[r]                                   &M \ar[r]                                  &0
}
\end{displaymath}
where $f:E\rightarrow F$ is an isomorphism (\cite[(14.8), p.275]{Demailly:2012}). This gives us the following short exact sequence
\begin{equation}
\label{equ:extensionF}
\quad 0 \to L \xrightarrow{f \circ \iota} F \xrightarrow{p \circ f^{-1}} M \to 0.
\end{equation}
Since $f$ is holomorphic, we have  $\bar{\partial}_{F}=f \circ \bar{\partial}_{E} \circ f^{-1}$. We define $h_{F}(\cdot,\cdot):=h_{E}(f^{-1}\cdot,f^{-1}\cdot)$. Then $D_{F}=f \circ D_{E} \circ f^{-1}$ is the Chern connection with respect to $\bar{\partial}_{F}$  and $h_{F}$, and $h_{F}$ serves as the Hermitian-Einstein metric on $F$.

\begin{lem}
\label{lem:well_defn_2}
According to the $\mathcal{C}^{\infty}$ isomorphism $(\iota^{*}\circ f^{-1})\oplus (p\circ f^{-1})\colon F\to L\oplus M$, the Chern connection $D_{F}$ can be expressed by
$$\displaystyle{D_{F} = \begin{pmatrix}
D_{L} & -\beta \\
\beta^{*} & D_{M}
\end{pmatrix}},$$ 
coinciding with  $D_E$ in \eqref{equ:matrix}.
\end{lem}

\begin{proof}
In line with the $\mathcal{C}^{\infty}$ isomorphism $(\iota^{*}\circ f^{-1})\oplus (p\circ f^{-1})\colon F\to L\oplus M$, we can represent the Chern connection $D_{F}$ as 
$$\displaystyle{ D_{F} = \begin{pmatrix} \widehat{D_{L}} & -\widehat{\beta} \\ \hat{\beta}^{*} & \widehat{D_{M}} \end{pmatrix}}.$$ Given that $h_{L}(\cdot,\cdot)=h_{E}(\iota\cdot,\iota\cdot)=h_{F}(f\circ \iota\cdot,f\circ \iota\cdot)$, the metrics induced by both $h_{E}$ and $h_{F}$ on $L$ are identical, leading to $D_{L}=\widehat{D_{L}}$ and $D_{M}=\widehat{D_{M}}$. Moreover, drawing from Theorem 14.3 d) in \cite[p.273]{Demailly:2012}, we deduce that $\bar{\partial}_{\operatorname{Hom}(M,E)}p^{*}=-\iota\circ \beta$ and $\bar{\partial}_{\operatorname{Hom}(M,F)}(f\circ p^{*})=-(f\circ\iota)\circ \widehat{\beta}$. Further, our calculations show that $$\bar{\partial}_{\operatorname{Hom}(M,F)}(f\circ p^{*})=f\bar{\partial}_{\operatorname{Hom}(M,E)}p^{*}+\big(\bar{\partial}_{\operatorname{Hom}(E,F)}f\big) p^{*}=f\circ(-\iota\circ \beta) =-(f\circ \iota)\circ \beta,$$ thereby concluding that $\widehat{\beta}=\beta$. Additionally, according to Theorem 14.3 a) in \cite[p.273]{Demailly:2012}, it is affirmed that $\partial_{\operatorname{Hom}(L,E)}\iota=p^{*}\circ \beta^{*}$ and $\partial_{\operatorname{Hom}(L,F)}(f\circ\iota)=(f\circ p^{*})\circ \widehat{\beta}^{*}$. Our calculation yields $$\partial_{\operatorname{Hom}(L,F)}(f\circ \iota)=f\partial_{\operatorname{Hom}(L,E)} \iota+\big(\partial_{\operatorname{Hom}(E,F)}f\big)\circ \iota= f\circ \partial_{\operatorname{Hom}(L,E)}\iota=(f\circ p^{*})\circ \beta^{*},$$ leading us to the conclusion that  $\widehat{\beta}^{*}=\beta^{*}$.
\end{proof}

According to \cite[Proposition 14.9]{Demailly:2012}, there exists a one-to-one correspondence  between the space of isomorphism classes of extensions of $L$ by $M$, denoted as ${\rm Ext}_{X}^{1}(M,L)$, and the cohomology group $H^{1}\big(X,{\rm Hom}(M,L)\big)$. Consider a short exact sequence
$
\displaystyle{\mathbb{E}:0\to L\stackrel{\iota}{\to} E\stackrel{p}{\to} M\to 0}
$,
where $E$ is a rank $2$ polystable vector bundle on $X$.  We employ a Hermitian-Einstein metric on $E$ and define $\beta^{*}$ as the second fundamental form of $L$ in $E$, with $\beta$ being the adjoint of $\beta^{*}$. Consequently, we establish the following mapping
\begin{align*}
{\rm Ext}_{X}^{1}(M,L)\supset {\rm Ext}_{X}^{1}(M,L)^{\rm ps}  &\to H^{1}(X,{\rm Hom}(M,L))\\
  \{\mathbb{E}\} &\mapsto \{\beta\},
\end{align*}
where ${\rm Ext}_{X}^{1}(M,L)^{\rm ps}$ represents the set of isomorphism classes of all polystable extensions in  ${\rm Ext}_{X}^{1}(M,L)$.
By utilizing Lemmas \ref{lem:well_defn_1} and \ref{lem:well_defn_2}, we establish a well-defined map as follows:
\begin{align*}
\Phi \colon {\rm Ext}_{X}^{1}(M,L)^{\rm ps} &\to H_{\bar{\partial}}^{1,0}\big(X, {\rm Hom}(L, M)\big) \\
  \{\beta\} &\mapsto \beta^{*}.
\end{align*}
Furthermore, for each $\lambda\in \mathbb{C}^{*}$, $\Phi$ maps the extension
\[
\mathbb{\lambda E}:\ 0\rightarrow L\stackrel{\iota/\lambda}{\longrightarrow}E\stackrel{p}{\rightarrow} M\rightarrow 0,
\]
to $(p\otimes {\rm id})\circ \partial_{E}\circ (\iota/\lambda)=\frac{\beta^{*}}{\lambda}$. As a result, $\Phi$ induces a map as follows, modulo the $\mathbb{C}^{*}$-action on both its domain and  target:
\begin{align*}
{\mathfrak R}_{(L, M)}\colon \mathbb{P}\Big({\rm Ext}_{X}^{1}(M,L)^{\rm ps}\Big) &\to \mathbb{P}\Big(H_{\bar{\partial}}^{1,0}\big(X, {\rm Hom}(L, M)\big)\Big)\\
  [\beta] &\mapsto [\beta^{*}]
\end{align*}
In this context, $\mathbb{P}\Big({\rm Ext}_{X}^{1}(M,L)^{\rm ps}\Big)$ and $\mathbb{P}\Big(H_{\bar{\partial}}^{1,0}\big(X, {\rm Hom}(L, M)\big)\Big)$ represents the projectivizations of ${\rm Ext}_{X}^{1}(M,L)^{\rm ps}$ and $H_{\bar{\partial}}^{1,0}\big(X, {\rm Hom}(L, M)\big)$, respectively. Furthermore, $[\beta]$ denotes the projectivization of  $\{\beta\}$.
In summary, this yields the following proposition:

\begin{prop}
\label{prop:rdm}
Consider the short exact sequence
$0\to L\to E\to M\to 0$, where $E$ is a polystable vector bundle on $X$.
Let's choose a Hermitian-Einstein metric on $E$ to define the second fundamental form $\beta^{*}$ of $L$ in $E$, and let $\beta$ be the adjoint of $\beta^{*}$. We can then define the ramification divisor map ${\mathfrak R}_{(L, M)}$ as follows:
\begin{align*}
{\mathfrak R}_{(L, M)}\colon \mathbb{P}\Big({\rm Ext}_{X}^{1}(M,L)^{\rm ps}\Big) &\to \mathbb{P}\Big(H_{\bar{\partial}}^{1,0}\big(X, {\rm Hom}(L, M)\big)\Big)\\
  [\beta] &\mapsto [\beta^{*}]
\end{align*}
This map is well-defined.
\end{prop}

\subsection{Fundamental lemma for reducible metrics}
In this subsection, we introduce the fundamental lemma (referenced as Lemma \ref{lemma:divisor}) concerning reducible metrics. This lemma is crucial for establishing a connection between reducible metrics and the previously mentioned ramification divisor map.
\begin{defn}
\label{defn:sp1}
Consider $X$ as a compact Riemann surface with genus $g_X \geq 0$, and let $L$ be a line bundle where $\deg L \leq 0$. A short exact sequence
\[0 \to L \to J \oplus J^{-1} \to L^{-1} \to 0, \quad J \in {\rm Pic}^0(X),\]
is termed {\it strictly polystable} if the section $f \colon X \to P := {\mathbb P}(J \oplus J^{-1})$ induced by the embedding $L \hookrightarrow J \oplus J^{-1}$ in this sequence is not locally flat. In other words, $f$ is non-constant on all open charts of $X$ corresponding to the projective unitary flat trivializations of $P$ (\cite[Definition 3.1.]{LSX:2021}). An extension in ${\rm Ext}^1(L^{-1}, L)$ is {\it strictly-polystable} if it is represented by a strictly polystable short exact sequence.
\end{defn}
For any strictly polystable short exact sequence of the form:
\[
\mathbb{E}: 0 \to L \stackrel{i}{\to} J \oplus J^{-1} \stackrel{p}{\to} L^{-1} \to 0, \quad J \in {\rm Pic}^0(X),
\]
we can select a Hermitian-Einstein metric $h$ on $J \oplus J^{-1}$. Consequently, we derive $\beta^{*} = (p \otimes {\rm id}) \circ \partial_{J \oplus J^{-1}} \circ i$, where $\partial_{J \oplus J^{-1}}$ denotes the $(1,0)$-part of the Hermitian-Einstein connection on $J \oplus J^{-1}$. In our subsequent analysis, we will demonstrate that $\beta^{*}$ acts as a holomorphic section of the line bundle $L^{-2} \otimes K_{X}$, and its associated divisor aligns with the effective divisor represented by the reducible metric defined by the aforementioned polystable extension.

\begin{lem}[Fundamental Lemma]
\label{lemma:divisor}
For a strictly polystable short exact sequence $0 \to L \stackrel{i}{\to} J \oplus J^{-1} \stackrel{p}{\to} L^{-1} \to 0$, employing a Hermitian-Einstein metric on $J \oplus J^{-1}$ yields $\beta^{*} \in H^{0}(X, L^{-2} \otimes K_X)$ as the second fundamental form of $L$. If $D$ denotes the effective divisor corresponding to the reducible metric for the embedding $L \stackrel{i}{\to} J \oplus J^{-1}$, then $\operatorname{Div}(\beta^{*})$, the associated divisor of $\beta^{*}$, is identical to $D$.
\end{lem}

\begin{proof}
It is important to note that the embedding $L \stackrel{i}{\to}J\oplus J^{-1}$ can be equivalently represented by a pair of sections $(i_1,\,i_2)$, where each belongs to $H^0\big({\rm Hom}(L,\,J)\big)$ and $H^0\big({\rm Hom}(L,\,J^{-1})\big)$ respectively, and notably, these sections do not share any common zero. The combined polystability and flatness of the bundle $J\oplus J^{-1}$ facilitate the construction of an open cover $\{U_{\alpha}\}$ of $X$. This cover is characterized by constant transition functions on the intersections $U_{\alpha\beta}=U_{\alpha}\cap U_{\beta}$, specifically defined as $\displaystyle{g_{\alpha\beta}=\begin{pmatrix} a_{\alpha\beta} & 0 \\ 0 & a_{\alpha\beta}^{-1} \end{pmatrix}}$, where $a_{\alpha\beta}\in\rm{U}(1)$.
Furthermore, the open covers can be chosen to be sufficiently small so that for each index $\alpha$, the corresponding set $U_{\alpha}$ contains at most one zero of the sections $i_{1}$ and $i_{2}$. Considering a holomorphic flat frame $(e_{\alpha},e^{*}_{\alpha})$ for the restricted bundle $(J\oplus J^{-1})|_{U_{\alpha}}$ in accordance with the transition functions $\{U_{\alpha\beta},g_{\alpha\beta}\}$, we can select a local holomorphic frame $l_{\alpha}$ for $L$ over $U_{\alpha}$, with the associated transition functions $\ell_{\alpha\beta}$ satisfying $l_{\beta}=\ell_{\alpha\beta}l_{\alpha}$. For the embedding $L\stackrel{i} \hookrightarrow J\oplus J^{-1}$, we express $i(l_{\alpha})$ as $s_{1,\alpha}e_{\alpha}+s_{2,\alpha}e^{*}_{\alpha}$, where $s_{1,\alpha}$ and $s_{2,\alpha}$ are devoid of common zeros. Consequently, on the overlap $U_{\alpha}\cap U_{\beta}$, the following relationships are derived:
\begin{align*}
s_{1,\alpha}e_{\alpha}+s_{2,\alpha}e^{*}_{\alpha}=&
i(l_{\alpha})=i(\ell^{-1}_{\alpha\beta}l_{\beta})=\ell^{-1}_{\alpha\beta}i(l_{\beta})\\
=&\ell^{-1}_{\alpha\beta}\big(s_{1,\beta}e_{\beta}+s_{2,\beta}e^{*}_{\beta}\big)
=\ell^{-1}_{\alpha\beta}s_{1,\beta}a_{\alpha\beta}e_{\alpha}+\ell^{-1}_{\alpha\beta}s_{2,\beta}a^{-1}_{\alpha\beta}e^{*}_{\alpha}
\end{align*}
From these expressions, it follows that $s_{1,\alpha}=\ell^{-1}_{\alpha\beta}a_{\alpha\beta}\cdot s_{1,\beta}$ and $s_{2,\alpha}=\ell^{-1}_{\alpha\beta}a^{-1}_{\alpha\beta}\cdot s_{2,\beta}$. Therefore, the assembled collection $\left\{{U_{\alpha},w_{\alpha}\colon=\frac{s_{1,\alpha}}{s_{2,\alpha}}}\right\}$ constitutes a global holomorphic section of $\mathbb{P}(J\oplus J^{-1})$, denoted by $w:X\rightarrow \mathbb{P}(J\oplus J^{-1})$, satisfying $w_{\alpha}=a^{2}_{\alpha\beta}\cdot w_{\beta}$. Given that $a_{\alpha\beta}\in U(1)$, and ${\rm d}s^{2}_{\alpha}$ is defined as the cone spherical metric constructed by $w^{*}_{\alpha}({\rm d}s_{0}^{2})$ on $U_{\alpha}$, where ${\rm d}s_{0}^{2}$ represents the Fubini-Study metric on the Riemann sphere, a globally defined cone spherical metric, denoted by ${\rm d}s^{2}$, is obtained. Furthermore, the construction of ${\rm d}s^{2}$ clearly indicates that the divisor $D$, represented by ${\rm d}s^{2}$, aligns with the ramification divisor $B_{\{U_{\alpha},w_{\alpha}\}}$ of the branched projective covering $\{U_{\alpha},w_{\alpha}\}$ of $X$.
Given that each $U_{\alpha}$ contains at most one zero point for the sections $i_{1}$ and $i_{2}$, the analysis bifurcates into two cases:

1. If $U_{\alpha}$ does not contain any zeros of $i_{2}$, then $s_{2,\alpha}$ has no zeros. In such a scenario, we define $\hat{l}_{\alpha}=\frac{l_{\alpha}}{s_{2,\alpha}}$, which also forms a holomorphic frame of $L$ on $U_{\alpha}$. With this definition, we have $i(\hat{l}_{\alpha})=w_{\alpha}e_{\alpha}+e_{\alpha}^*$, leading to the expression:
\begin{align*}
\beta^{*}(\hat{l}_{\alpha})=&(p\otimes {\rm id})\circ \partial_{J\oplus J^*}\circ i(\hat{l}_{\alpha})
=(p\otimes {\rm id})\circ \partial_{J\oplus J^*}(w_{\alpha}e_{\alpha}+e_{\alpha}^*)\\
=&(p\otimes {\rm id})(dw_{\alpha}\otimes e_{\alpha})
=dw_{\alpha}\otimes p(e_{\alpha}),
\end{align*}
where $p(e_{\alpha})$ serves as a holomorphic frame of $L^{-1}$ on $U_{\alpha}$.

2. If $U_{\alpha}$ contains no zeros of $i_{1}$, then $s_{1,\alpha}$ has no zeros. In this case, we define $\tilde{l}_{\alpha}=\frac{l_{\alpha}}{s_{1,\alpha}}$, also a holomorphic frame of $L$ on $U_{\alpha}$. Consequently, we have $i(\tilde{l}_{\alpha})=e_{\alpha}+\frac{1}{w_{\alpha}}e_{\alpha}^*$, and hence:
\begin{align*}
\beta^{*}(\tilde{l}_{\alpha})=&(p\otimes {\rm id})\circ \partial_{J\oplus J^*}\circ i(\tilde{l}_{\alpha})
=(p\otimes {\rm id})\circ \partial_{J\oplus J^*}\bigg(e_{\alpha}+\frac{1}{w_{\alpha}}e_{\alpha}^*\bigg)\\
=&(p\otimes {\rm id})\Bigg(d\left(\frac{1}{w_{\alpha}}\right)\otimes e_{\alpha}^*\Bigg)
={\rm d}\left(\frac{1}{w_{\alpha}}\right)\otimes p(e_{\alpha}^*),
\end{align*}
where $p(e_{\alpha})^*$ represents a holomorphic frame of $L^{-1}$ on $U_{\alpha}$.

In both cases, the associated divisor $\operatorname{Div} (\beta^{*})$ of $\beta^{*}$ is shown to be identical to the ramification divisor $B_{\{U_{\alpha},w_{\alpha}\}}$. 
\end{proof}

\section{Proof of Theorem \ref{thm:irr}}

Building on the foundational work of Bartolucci, De Marchis, and Malchiodi (\cite{BdeMM:2011}), as well as Troyanov (\cite{Troyanov:1991}), it is established that for any effective divisor $D$ with an odd degree exceeding $2g_{X} - 2$, there invariably exists an irreducible metric that represents $D$. In this section, we aim to prove Theorem \ref{thm:irr} by employing Brouwer's invariance of domain theorem on the surjective map $\sigma$ referenced in \eqref{equ:surmap}. 
This method substantiates the hypothesis that irreducible metrics, which represent effective divisors of an odd degree greater than $2g_X$, are {\it not} unique.

\begin{lem}
\label{lem:inj}
Let $L$ and $M$ be two line bundles on a compact Riemann surface $X$  with $\deg L < \deg M$, such that there exists a stable extension of $M$ by $L$. The restriction of the map $\sigma$
defined in 
\eqref{equ:surmap} to the projectivization $\mathbb{P}\left(\operatorname{Ext}_{X}^1(M, L)^{s}\right)$ of the space $\operatorname{Ext}_{X}^1(M, L)^{s}$, comprising stable extensions of $M$ by $L$, is an injective map into the space $\mathcal{MI}(X, \mathbb{Z})$ of irreducible metrics representing effective divisors on $X$.
\end{lem}

\begin{proof}
According to \cite[Theorem 2.2]{LSX:2021}, the space $\operatorname{Ext}_{X}^1(M, L)^{s}$ is Zariski open in $\operatorname{Ext}_{X}^1(M, L)$, and so is $\mathbb{P}\left(\operatorname{Ext}_{X}^1(M, L)^{s}\right)$ in $\mathbb{P}\left(\operatorname{Ext}_{X}^1(M, L)\right)$. Suppose that there exist two stable extensions
\[ ({\mathbb E}_1): 0\to L\to E_1\to M\to 0\quad \text{and}\quad ({\mathbb E}_2): 0\to L\to E_2\to M\to 0 \]
which define the same irreducible cone spherical metric. Utilizing \cite[Theorem 3.1]{LSX:2021}, we obtain ${\mathbb P}(E_1)={\mathbb P}(E_2)$, denoted by $P$. Let $s_{i} \colon X \to P$ be the section of $P$ induced by $L\hookrightarrow E_i$ ($i=1,\, 2$). Then, there exists a line bundle $T$ over $X$ such that there is an isomorphism $\eta\colon E_1\to E_2\otimes T$ since $\mathbb{P}(E_{1}) = \mathbb{P}(E_{2})$ (see \cite[p.515]{GH:1994}). Moreover, $T$ is a 2-torsion  since $\det E_{1} = \det E_{2}$, and $s_{2}$ is the image of $s_{1}$ under the automorphism of $P$ induced by $\eta$, as $\mathbb{E}_{1}$ and $\mathbb{E}_{2}$ define the same cone spherical metric. This leads to the following commutative diagram:

\[ \xymatrix{
0\ar[r] &L\ar[r]\ar[d]^{\eta_{|L}} &E_{1}\ar[d]^{\eta} \\
0\ar[r] &L\otimes T \ar[r] &E_{2} \otimes T,
}\]
where we view the equivalence class $(\mathbb{E}_{2})$ as an element in $\operatorname{Ext}_{X}^{1}(M\otimes T, L\otimes T)$ represented by $0\to L\otimes T \to E_2\otimes T \to M \otimes T \to 0$, via the following isomorphisms 
$$\operatorname{Ext}_{X}^1(M, L) \cong H^{1}(X, L\otimes M^{-1}) \cong H^{1}\left(X, (L\otimes T) \otimes (M\otimes T)^{-1}\right) \cong \operatorname{Ext}_{X}^1(M\otimes T, L \otimes T).$$
By examining the commutative diagram, we deduce that $L\cong \eta(L) = L \otimes T$. Consequently, $T \cong \mathcal {O}_X$ yields the identity map of $P$. As a result, the two line subbundles $L\hookrightarrow E_1$ and $L\hookrightarrow E_2$ induce the same section of $P\to X$, i.e., $s_{1} = s_{2}$, and the two extensions $({\mathbb E}_1)$ and $({\mathbb E}_2)$ are co-linear in $\operatorname{Ext}_{X}^1(M,L)$.
\end{proof}

\begin{proof}[Proof of Theorem \ref{thm:irr}]

Choose the line bundle $M$ on $X$ such that $M=K_X^{-1}\otimes {\mathcal O}_X(D)$. Then $d:=\deg M$ is an odd integer $\geq 3$. The proof  involves the following three steps.

Firstly, by \cite[Theorem 2.4]{LSX:2021} and \cite[Theorem 1.1]{BdeMM:2011}, the ramification divisor map $\mathfrak{R}_{(\mathcal {O}_X,M)}$ is a \emph{real analytic} \emph{surjective} map from the projectivization $\mathbb{P}\left(\operatorname{Ext}_{X}^1(M,\mathcal{O}_X)^{s}\right)$ of the space $\operatorname{Ext}_{X}^1(M,\mathcal {O}_X)^{s}$ to $|K_X\otimes M|=|D|$. By the Riemann-Roch theorem, the rational map
     \begin{eqnarray*}
     \varphi=\varphi_{K_X\otimes M} \colon X&\to& {\mathbb P}\left(\operatorname{Ext}_{X}^1(M,\mathcal{O}_X)\right)={\mathbb P}\Bigl(\bigl(H^0(X,\, K_X\otimes M)\bigr)^*\Bigr)\\
     x&\mapsto& \left\{\tau \in H^0(X,\, K_X\otimes M) \mid \tau(x)=0\right\}
     \end{eqnarray*}
     becomes a non-degenerate embedding of $X$ into ${\mathbb P}\left(\operatorname{Ext}_{X}^1(M,\mathcal{O}_X)\right)$. Utilizing \cite[Proposition 1.1]{LN:1983}, a nontrival extension $0\to \mathcal {O}_X\to E\to M\to 0$ is unstable if and only if it lies in the secant variety $\operatorname{Sec}_{(d-1)/2}\bigl(\varphi(X)\bigr)$, which has dimension $\frac{d-1}{2}>0$. Thus, the domain $\mathbb{P}\left(\operatorname{Ext}_{X}^1(M,\mathcal{O}_X)^{s}\right)$ of the ramification divisor map ${\mathfrak R}_{({\mathcal O}_X, M)}$ is a Zariski open and proper subset of ${\mathbb P}\left(\operatorname{Ext}_{X}^1(M,\mathcal{O}_X)\right)$. In particular, it is non-compact.

Secondly, \emph{${\mathfrak R}_{(\mathcal {O}_X, M)} \colon \mathbb{P}\left(\operatorname{Ext}_{X}^1(M,\mathcal{O}_X)^{s}\right) \to |D|$ must not be one-to-one.} Suppose that it is injective. Then, since it is continuous, ${\mathfrak R}_{(\mathcal{O}_X,M)}$ is open by Brouwer's invariance of domain theorem \cite{Brouwer:1912}. As it is also surjective, $\mathfrak {R}_{(\mathcal {O}_X, M)}$ would establish a homeomorphism between $\mathbb{P}\left(\operatorname{Ext}_{X}^1(M,\mathcal{O}_X)^{s}\right)$ and the complete linear system $|D|$. This contradicts the non-compactness of $\mathbb{P}\left(\operatorname{Ext}_{X}^1(M,\mathcal{O}_X)^{s}\right)$.

Thirdly, following the result of the second step, there exists a divisor $D' \in |D|$ with at least two distinct pre-images, namely
\[ ({\mathbb E}_1): 0\to {\mathcal O}_X\to E_1\to M\to 0\quad \text{and}\quad ({\mathbb E}_2): 0\to {\mathcal O}_X\to E_2\to M\to 0,\]
under  ${\mathfrak R}_{({\mathcal O}_X,\, M)}$. Hence, these two stable extensions $({\mathbb E}_1)$ and $({\mathbb E}_2)$ are not co-linear in $\operatorname{Ext}_{X}^1(M,\,{\mathcal O}_X)$. By Lemma \ref{lem:inj}, they give two non-isometric irreducible cone spherical metrics representing $D'$ on $X$.
\end{proof}

\section{Strictly polystable extensions and integral reducible metrics}

In this section, we initiate a thorough analysis of the existence problem associated with strictly polystable extensions. This foundational theoretical framework is essential for our upcoming detailed examination of reducible metrics, which we will explore in the subsequent section. Moreover, we demonstrate that various projectivizations of strictly polystable extension classes result in uniquely distinct, yet projectively equivalent, classes of reducible metrics, as elaborated in Proposition \ref{prop:red}.

\subsection{Existence of strictly polystable extensions}

\begin{lem}
\label{lem:negative}
 Let $X$ be a compact Riemann surface of genus $g_X \geq 0$ and $L$ a line bundle on $X$. Then there exists a strictly polystable extension of $L^{-1}$ by $L$ solely when $\deg\, L < 0$. Furthermore, if $\deg\, L < 0$ and $J \in \operatorname{Pic}^{0}(X)$, then a short exact sequence $0\to L\to J \oplus J^{-1} \to L^{-1}\to 0$ must be strictly polystable.
\end{lem}
\begin{proof} This lemma comprises two statements, both of which we shall prove sequentially.

Let $J$ be a line bundle of degree zero on $X$, i.e. $J \in \operatorname{Pic}^{0}(X)$ such that there exists an embedding $L \hookrightarrow J \oplus J^{-1}$, and ensure that the induced section $f \colon X \to P:=  \mathbb{P}(J \oplus J^{-1})$ is not locally flat. Observe that the embedding $L \hookrightarrow J \oplus J^{-1}$ is equivalent to a pair $(s_1,\,s_2)$ of two sections $s_1,\, s_2$ lying in $H^0 \big(X, \operatorname{Hom}(L,\,J)\big)$, $H^0\big(X, \operatorname{Hom}(L,\,J^{-1})\big)$ respectively such that these two sections possess no common zero. Consequently, $\deg L \leq 0$; otherwise, $\deg \operatorname{Hom}(L,J) < 0$ and $\deg \operatorname{Hom}(L,J^{-1}) < 0$ would imply $s_{1}=0=s_{2}$. If $\deg L = 0$, then $\deg \operatorname{Hom}(L,J) = 0 = \deg \operatorname{Hom}(L,J^{-1})$. This implies that $s_j \, (j=1,2)$ either vanishes nowhere or vanishes identically on $X$. Without loss of generality, we assume that $s_1$ vanishes nowhere. Consequently, $L\cong J$.
\begin{itemize}
\item If $J$ is a 2-torsion line bundle, i.e. $J \cong J^{-1}$, then $s_1$ and $s_2$ are co-linear. In this case, the section $f$ induced by the embedding $(s_1,\, s_2) \colon L \to J\oplus J^{-1}$ is locally constant.
\item If $J$ is not a 2-torsion line bundle, then $L\cong J \not\cong J^{-1}$ and $s_2$ vanishes identically. Hence the embedding $(s_1,0) \colon L \to J \oplus J^{-1}$ gives the locally constant section
\[ f \colon X \to\mathbb{P}(J\oplus J^{-1}),\quad f \equiv[1:\,0]. \]
\end{itemize}
In both cases, the induced section $f$ is consistently locally flat, which contradicts the assumption that $f$ is not locally flat. Therefore, $\deg L$ cannot be zero. In summary, if there exists a strictly polystable extension of $L^{-1}$ by $L$, then $\deg L < 0$.

For the second assertion, by using \cite[Lemma 3.2.]{LSX:2021}, we know  that the section $f \colon X\to P:={\mathbb P}(J\oplus J^{-1})$
has ramification divisor $B_f$ lying in the complete linear system $|L^{-2}\otimes K_X|$.
\begin{itemize}

\item If $\deg\, B_f=2g_X-2-2\deg\, L > 0$, then $f$ is not locally flat.

\item If $\deg\, B_f=0$, then $g_X=0$, $L \cong \mathcal{O}_{\mathbb{P}^{1}}(-1)$ and
$J$ is a trivial line bundle. Then, the sections $s_{1}, s_{2} \in H^{0}\big(\mathbb{P}^{1}, \mathcal{O}_{\mathbb{P}^{1}}(1)\big)$ possess no common zero. Hence, $s_{1}$ and $s_{2}$ cannot vanish identically and be co-linear. In other words, the induced section $f \colon \mathbb{P}^1\to \mathbb{P}\big({\mathcal{O}_{\mathbb{P}^{1}}} \oplus {\mathcal{O}_{\mathbb{P}^{1}}}\big)=\mathbb{P}^1\times \mathbb{P}^1$ is not locally flat.         \qedhere

\end{itemize}
\end{proof}

By virtue of this lemma, we can legitimately refine Definition \ref{defn:sp1} succinctly as follows.

\begin{defn}
\label{defn:sp2}
Let $X$ be a compact Riemann surface of genus $g_X \geq 0$ and $L$ a line bundle of $\deg\,L<0$.
We call the short exact sequence
\[0\to L\to J\oplus J^{-1}\to L^{-1}\to 0,\quad J\in {\rm Pic}^0(X),\]
 {\it strictly polystable}.  An element in ${\rm Ext}^1_{X}(L^{-1},\,L)$ is called {\it strictly polystable} if it could be represented by a strictly polystable extension of $L^{-1}$ by $L$.
\end{defn}

\begin{ques}
\label{ques:sp}
Let $X$ be a compact Riemann surface of genus $g_X \geq 0$ and $L$ a line bundle with $\deg\,L<0$.
\begin{enumerate}
\item Determine the necessary and sufficient condition for the existence of a strictly polystable extension of $L^{-1}$ by $L$.
\item Characterize the algebraic subset consisting of all strictly polystable extensions in the complex vector space
$\operatorname{Ext}^1_{X}(L^{-1},\, L)$.
\end{enumerate}
\end{ques}

\begin{exmp}
{\rm  We shall initially address this question on $\mathbb{P}^1$. }

\begin{enumerate}

\item There exists a strictly polystable extension
if and only if $-d:=\deg\, L<0$.
\item The strictly polystable extensions of $L^{-1}$ by $L$ form a Zariski open subset $\big(\operatorname{Ext}^1_{\mathbb{P}^{1}}(L^{-1},\, L)\big)^{\rm sp}$ of $\operatorname{Ext}^1_{\mathbb{P}^{1}}(L^{-1},\, L)$.
Moreover, the projectivization ${\mathbb P}\left(\big(\operatorname{Ext}_{\mathbb{P}^{1}}^1(L^{-1},\, L)\big)^{\rm sp}\right)$ corresponds to the space of all the rational functions of degree $d$ moduli the post-compositions of M\" obius transformations, which is a Zariski open subset in the Grassmannian ${\rm Gr}_2({\mathbb C}^{d+1})$.
In particular, ${\mathbb P}\left(\big(\operatorname{Ext}_{\mathbb{P}^1}^1(L^{-1},\, L)\big)^{\rm sp}\right)$ has dimension $2d-2$.
\end{enumerate}
\begin{proof} We prove the two statements sequentially.
\begin{enumerate}
\item Recall that $H^0\big(\mathbb{P}^{1}, {\mathcal O}(d)\big)$ coincides with the space of degree $d$ homogeneous polynomials in $z_0$ and $z_1$, where $[z_0:z_1]$ is the homogeneous coordinate on ${\mathbb P}^1$ (\cite[pp. 175-177]{Kod:2005}). Since $-d=\deg\, L<0$, any two homogeneous polynomials $f(z_0,,z_1)$ and $g(z_0,,z_1)$ of degree $d$ that have no common zero provide an embedding of $L$ into ${\mathcal O}\oplus {\mathcal O}$ and induce a section $f \colon {\mathbb P}^1\to {\mathbb P}^1\times {\mathbb P}^1$ corresponding to the rational function $f/g$ on ${\mathbb P}^1$ of degree $d$.

\item It follows from (1) since we could identify ${\mathbb C}^{d+1}$ with the space of degree $d$ homogeneous polynomials in $z_0$ and $z_1$.  \qedhere
\end{enumerate}
\end{proof}
\end{exmp}

Then, we give a comprehensive solution for elliptic curves as follows, providing a partial answer to Question \ref{ques:sp}.
Initially, we establish a lemma applicable to general Riemann surfaces.
\begin{lem}
\label{lemma:ample}
Let $L$ and $J$ be two line bundles over a compact Riemann surface $X$ of $g_X \geq 0$
such that $\deg\, L\leq -2g_X$ and $\deg\, J=0$. Then
the rank two vector bundle $L^{-1}\otimes (J\oplus J^{-1})$ is ample in the sense of Atiyah (\cite[p.417]{Atiyah:1957elliptic}). In particular, there is an embedding $L\hookrightarrow J\oplus J^{-1}$.
\end{lem}
\begin{proof}
 Since $\deg (L^{-1}\otimes J) = k \geq 2g_{X}$, both $L^{-1}\otimes J$ and $L^{-1}\otimes J^{-1}$ are ample
 in the sense of Atiyah (\cite[p.417]{Atiyah:1957elliptic})
 by Lemma 8 in \cite{Atiyah:1957elliptic}.
By using the splitting exact sequence of vector bundles
\[
0\rightarrow L^{-1}\otimes J\rightarrow L^{-1}\otimes (J\oplus J^{-1})\rightarrow L^{-1}\otimes J^{-1}\rightarrow 0,
\]
and Lemma 1 (i) in \cite{Atiyah:1957elliptic}, we could obtain that $L^{-1}\otimes (J\oplus J^{-1})$ is ample.
In particular, $L^{-1}\otimes (J\oplus J^{-1})$ is globally generated, i.e.
the evaluation map
\[H^0\big(X, L^{-1}\otimes (J\oplus J^{-1})\big)\to \big(L^{-1}\otimes (J\oplus J^{-1})\big)_x,\quad
s\mapsto s(x)\]
is surjective at each point $x\in X$. By counting dimension, the holomorphic sections of $L^{-1}\otimes (J\oplus J^{-1})$  with at least one zero form a proper subvariety of $H^0\big(X, L^{-1}\otimes (J\oplus J^{-1})\big)$.
Therefore, $L^{-1}\otimes (J\oplus J^{-1})$ has a nowhere vanishing section, i.e. ${\mathcal O}_X$ could be embedded into  $L^{-1}\otimes (J\oplus J^{-1})$.
\end{proof}


\begin{thm}
\label{thm:inclusionindirectsum}
Let $X$ be an elliptic curve, $L$ a line bundle of negative degree over $X$ and $J$ a line bundle of degree zero over $X$. Then there exists an embedding of $L$ into $J \oplus J^{-1}$ if either one of the two conditions holds
\begin{itemize}
\item $\deg L \leq -2$;
\item $\deg L = -1$ and $J^{2} \neq \mathcal{O}_{X}$.
\end{itemize}
In particular, there always exists a strictly polystable extension of $L^{-1}$ by $L$.

\end{thm}

\begin{proof}
Let $\deg L = -k < 0$. Since $g_{X}=1$, by Lemma \ref{lemma:ample}, there exists an embedding of $L$ into $J\oplus J^{-1}$ for $\deg L = -k \leq -2$. If $\deg L =-1$, we pick $J\in \rm{Pic}^{0}(X)$ with $J^{2}\neq \mathcal{O}_{X}$. By the Riemann-Roch theorem, we have
\[
h^{0}(L^{-1}\otimes J)-h^{1}(L^{-1}\otimes J)=1-g_{X}+1=1,
\]
where $h^{i}(L^{-1}\otimes J)=\dim_{\mathbb{C}}H^{i}(X, L^{-1}\otimes J)\ {\rm for}\ i=0,1$. Hence $h^{0}(L^{-1}\otimes J)\geq1$. There exist $0\not\equiv s_{1}\in H^{0}(X, L^{-1}\otimes J)$ and $0\not\equiv s_{2}\in H^{0}(X, L^{-1}\otimes J^{-1})$, both of which have a single zero on $X$. Suppose that both $s_1$ and $s_2$ vanish at the same point $x$ on $X$. Then $L^{-1}\otimes J\cong \mathcal{O}(x)\cong L^{-1}\otimes J^{-1}$ and $J\cong J^{-1}$. Contradiction. Hence,  $(s_{1}, s_{2})\in H^{0}\big(X, L^{-1}\otimes (J\oplus J^{-1})\big)$ is a nowhere vanishing section which induces an embedding $L\hookrightarrow J\oplus J^{-1}$.

\end{proof}

\subsection{Strictly polystable extensions and integral reducible metrics}

\begin{defn}
Two cone spherical metrics are called {\it projectively equivalent} if and only if their developing maps have the difference by a post-composition of
some M\" obius transformation.
\end{defn}

\begin{lem} {\rm (\cite[Fact 2.1., p.77]{Umehara-Y:2000})}  Any two projectively equivalent irreducible metrics are isometric.
\end{lem}

Utilizing the surjective map $\sigma$ in (\ref{equ:surmap}), we can derive a cone spherical metric that represents an effective divisor through a polystable extension of two line bundles. When we restrict the domain of $\sigma$ to the space ${\rm Ext}^1(L^{-1},L)^{\rm sp}$, encompassing strictly polystable extensions of $L^{-1}$ by $L$, we obtain reducible cone spherical metrics. This leads to the potential formation of a canonical map
\[
\bar{\sigma}\colon{\mathbb P}\bigl({\rm Ext}^1(L^{-1},\,L)^{\rm sp}\bigr)\rightarrow \mathcal{MR}(X,\,{\mathbb Z})
\]
where ${\mathbb P}\bigl(H^1(X,\, L^2)^{\rm sp}\bigr)={\mathbb P}\bigl({\rm Ext}^1(L^{-1},\,L)^{\rm sp}\bigr)$ signifies the projectivization of the space ${\rm Ext}^1(L^{-1},\,L)^{\rm sp}$. Here, $\mathcal{MR}(X,\,{\mathbb Z})$ denotes the space of projectively equivalent classes of reducible metrics that represent effective divisors on $X$. It is crucial to validate the well-defined nature of this map, particularly its independence from the choice of representations in an element of ${\mathbb P}\bigl({\rm Ext}^1(L^{-1},L)^{\rm sp}\bigr)$. The substantiation of this aspect will be addressed in Lemma \ref{lem:well-defined}. Additionally, we aim to demonstrate the injectivity of the map $\bar{\sigma}$.
In conclusion, this leads us to the formulation of the following proposition.

\begin{prop}
\label{prop:red}
Let $L$ be a line bundle on a compact Riemann surface $X$ with $\deg\, L<0$ such that there exists a strictly polystable extension of $L^{-1}$ by $L$.
Subsequently, we introduce a canonical mapping
\[
\bar{\sigma}\colon{\mathbb P}\bigl({\rm Ext}^1(L^{-1},\,L)^{\rm sp}\bigr)\rightarrow \mathcal{MR}(X,\,{\mathbb Z})
\]
where  ${\mathbb P}\bigl({\rm Ext}^1(L^{-1},\,L)^{\rm sp}\bigr)$ represents the projectivization of the space ${\rm Ext}^1(L^{-1},\,L)^{\rm sp}$ of strictly polystable extensions of $L^{-1}$ by $L$. Then this map is well-defined and injective.
\end{prop}

So this proposition shows that the projectivization of a strictly polystable extension class corresponds to a projectively equivalent class of reducible metrics. In the subsequent discussion, we will divide the proof of this proposition into the following two lemmas.

\begin{lem}
\label{lem:well-defined}
The map $\bar{\sigma}$ is well-defined.
\end{lem}

\begin{proof}

We consider an embedding $L\stackrel{i} \hookrightarrow J\oplus J^{-1}$, where $J$ denotes a holomorphic line bundle over $X$ with $\deg J=0$.

{\it Case 1.} Assume that $J$ is not a 2-torsion, i.e., $J^{2}\neq{\mathcal O}_X$. Then the automorphism group of the rank two holomorphic vector bundle 
$J\oplus J^{-1}\to X$ can be expressed as
\[
{\rm Aut}_{X}(J\oplus J^{-1})=
\left\{\phi=\begin{pmatrix} \lambda & 0 \\ 0 & \mu \end{pmatrix}\in \operatorname{Hom}(J\oplus J^{-1},J\oplus J^{-1}): \lambda,\mu\in \mathbb{C}^* \right\}.
\]
Thus, due to the polystability and flatness of $J\oplus J^{-1}$, there exists an open cover $\{U_{\alpha}\}$ of $X$ and constant transition functions on $U_{\alpha\beta}=U_{\alpha}\cap U_{\beta}$ given by
$$
\displaystyle{g_{\alpha\beta}=\begin{pmatrix} a_{\alpha\beta} & 0 \\ 0 & a_{\alpha\beta}^{-1} \end{pmatrix}},
$$
where  $a_{\alpha\beta}\in\rm{U}(1)$.
Let $(e_{\alpha},e^{*}_{\alpha})$ be a holomorphic frame of $(J\oplus J^{-1})|_{U_{\alpha}}$ relative to $\{U_{\alpha\beta},g_{\alpha\beta}\}$, i.e.
$$\begin{pmatrix}
  {e}_{\beta} \\
  {e}^{*}_{\beta}
\end{pmatrix} = \begin{pmatrix}
  a_{\alpha\beta} & 0 \\
  0 & a^{-1}_{\alpha\beta}
\end{pmatrix} \begin{pmatrix}
 {e}_{\alpha} \\
  {e}^{*}_{\alpha}
\end{pmatrix}.$$
Since $L$ is a holomorphic line subbundle of $J\oplus J^{-1}$, we choose a holomorphic frame $l_{\alpha}$ of $L$ on $U_{\alpha}$, and denote the transition functions as $\ell_{\alpha\beta}$. This means that $\ell_{\alpha\beta}\colon U_{\alpha}\cap U_{\beta}\rightarrow GL(1;\mathbb{C})$ are holomorphic maps and $l_{\beta}=\ell_{\alpha\beta}l_{\alpha}$ holds.  Since $L\stackrel{i} \hookrightarrow J\oplus J^{-1}$ is a holomorphic embedding,  $i(l_{\alpha})$ is a nowhere vanishing section of $J\oplus J^{-1}$ on $U_{\alpha}$. We set $i(l_{\alpha})=s_{1,\alpha}e_{\alpha}+s_{2,\alpha}e^{*}_{\alpha}$, where $s_{1,\alpha}$ and $s_{2,\alpha}$ have no common zeros. On $U_{\alpha}\cap U_{\beta}$, we have $$s_{1,\alpha}=\ell^{-1}_{\alpha\beta}a_{\alpha\beta}\cdot s_{1,\beta}, \quad s_{2,\alpha}=\ell^{-1}_{\alpha\beta}a^{-1}_{\alpha\beta}\cdot s_{2,\beta}.$$ Thus, $\left\{{U_{\alpha},w_{\alpha}:=\frac{s_{1,\alpha}}{s_{2,\alpha}}}\right\}$ defines a global holomorphic section of $\mathbb{P}(J\oplus J^{-1})$ denoted by $w:X\rightarrow \mathbb{P}(J\oplus J^{-1})$, and $w_{\alpha}=a^{2}_{\alpha\beta}\cdot w_{\beta}$. It is worth noting that  $a_{\alpha\beta}\in U(1)$. 
Then we obtain from ${\rm d}s^{2}_{\alpha}:=w^{*}_{\alpha}({\rm d}s_{0}^2)$ on $U_\alpha$ a cone spherical metric on $X$, denoted by ${\rm d}s^2$.
We now turn to another embedding $L\stackrel{\phi\circ i} \hookrightarrow J\oplus J^{-1}$, where $\phi\circ i(l_{\alpha})=\phi(s_{1,\alpha}e_{\alpha}+s_{2,\alpha}e^{*}_{\alpha})=\lambda s_{1,\alpha}e_{\alpha}+\mu s_{2,\alpha}e^{*}_{\alpha}$. Here,
$$
\displaystyle{\left\{{U_{\alpha},\tilde{w}_{\alpha}=\frac{\lambda}{\mu}\cdot w_{\alpha}=\frac{\lambda}{\mu}\cdot\frac{s_{1,\alpha}}{s_{2,\alpha}}}\right\}}
$$
defines a global holomorphic section of $\mathbb{P}(J\oplus J^{-1})$ denoted by $\tilde{w}=\frac{\lambda}{\mu}\cdot w:X\rightarrow \mathbb{P}(J\oplus J^{-1})$. We obtain a spherical reducible metric ${\rm d}\tilde{s}^2$. It is worth noting that ${\rm d}\tilde{s}^2$ and ${\rm d}s^{2}$ are projectively equivalent reducible metrics representing the same divisor.

{\it Case 2.} Suppose that $J$ is 2-torsion, i.e., $J^{2}={\mathcal O}_X$. Then the automorphism group of $J\oplus J^{-1}$ is given by
$$
\displaystyle{{\rm Aut}_{X}(J\oplus J^{-1})=
\left\{\psi=\begin{pmatrix} a & b \\ c & d \end{pmatrix}\in \operatorname{Hom}(J\oplus J^{-1},J\oplus J^{-1}): a,\ b,\ c,\ d\in \mathbb{C},\,ad-bc\not=0 \right\}}.
$$ Since $J=J^{-1}$, we can choose an open cover $\{U_{\alpha}\}$ of $X$ and transitions functions
$
\displaystyle{g_{\alpha\beta}=\begin{pmatrix} a_{\alpha\beta} & 0 \\ 0 & a_{\alpha\beta} \end{pmatrix}}
$
on $U_{\alpha\beta}=U_{\alpha}\cap U_{\beta}$,
where  $a_{\alpha\beta}\in\rm{U}(1)$.
Let $(e_{\alpha},e^{*}_{\alpha})$ be a holomorphic frame of $(J\oplus J^{-1})|_{U_{\alpha}}$ relative to $\{U_{\alpha\beta},g_{\alpha\beta}\}$.
We choose a holomorphic frame $l_{\alpha}$ of $L$ on $U_{\alpha}$. Consider the embedding $L\stackrel{ i} \hookrightarrow J\oplus J^{-1}$ and set $i(l_{\alpha})=s_{1,\alpha}e_{\alpha}+s_{2,\alpha}e^{*}_{\alpha}$. For the embedding $L\stackrel{\psi\circ i} \hookrightarrow J\oplus J^{-1}$, the transformation $\psi$ maps the frame $(e_{\alpha},e^{*}_{\alpha})$ to $(ae_{\alpha}+ce^{*}_{\alpha},be_{\alpha}+de^{*}_{\alpha})$, i.e.,
\[
\psi(e_{\alpha})=ae_{\alpha}+ce^{*}_{\alpha},\quad  \psi(e^{*}_{\alpha})=be_{\alpha}+de^{*}_{\alpha}.
\]
Then, we have
$$
\psi\circ i(l_{\alpha})=\psi(s_{1,\alpha}e_{\alpha}+s_{2,\alpha}e^{*}_{\alpha})
=(as_{1,\alpha}+bs_{2,\alpha})e_{\alpha}+(cs_{1,\alpha}+ds_{2,\alpha})e^{*}_{\alpha}.
$$
Hence,
$$
\displaystyle{\left\{{U_{\alpha\beta},\hat{w}_{\alpha}=\frac{as_{1,\alpha}+bs_{2,\alpha}}{cs_{1,\alpha}+ds_{2,\alpha}}}=\psi(w_{\alpha})\right\}}
$$
defines a global holomorphic section of $\mathbb{P}(J\oplus J^{-1})$ denoted by $\hat{w}=\psi(w)=\frac{aw+b}{cw+d}\colon X\rightarrow \mathbb{P}(J\oplus J^{-1})$. We obtain a spherical reducible metric ${\rm d}\hat{s}^{2}$. The metrics ${\rm d}\hat{s}^{2}$ and ${\rm d}s^{2}$ are projectively equivalent reducible metrics representing the same divisor.
\end{proof}

\begin{lem}
The map $\bar{\sigma}$ is injective.
\end{lem}

\begin{proof}
The proof can be divided into the following two cases.

{\it Case 1.} Let $g_{1}$ and $g_{2}$ be two projectively equivalent reducible metrics. Assume that $g_{2}$ is a nontrivial reducible metric. Let $f_{1}$ and $f_{2}$ be the developing maps of $g_{1}$ and $g_{2}$,  respectively, with monodromies belonging to $U(1)$.  We can express $f_{2}$ as
$$
f_{2}=\frac{af_{1}+b}{cf_{1}+d},
$$
where $a, b, c, d\in\mathbb{C}$ and $ad-bc=1$. We choose a function element  \ $\mathfrak{f}_{2}$ of $f_2$ near  $p\in X$ such that
$$
\mathfrak{f}_{2}=\frac{a\mathfrak{f}_{1}+b}{c\mathfrak{f}_{1}+d},
$$
where ${\mathfrak f}_1$ is some function element of $f_1$.
Since $g_{2}$ is a nontrivial reducible metric and the monodromy of $f_{2}$ belongs to $U(1)$, there exist
$\varphi$, $\theta\in \mathbb{R}\backslash {\mathbb Z}$ such that
\[
e^{2\pi\sqrt{-1}\varphi}\mathfrak{f}_{2}= e^{2\pi\sqrt{-1}\varphi} \frac{a\mathfrak{f}_{1}+b}{c\mathfrak{f}_{1}+d} = \frac{a e^{2\pi\sqrt{-1}\theta}\mathfrak{f}_{1}+b}{ce^{2\pi\sqrt{-1}\theta} \mathfrak{f}_{1}+d}.
\]
This is equivalent to the system
$$
\displaystyle{\left\{\  \begin{aligned} ac e^{2\pi\sqrt{-1}\theta}(1-e^{2\pi\sqrt{-1}\varphi}) &=0 \\
  e^{2\pi\sqrt{-1}\varphi}(ad+bce^{2\pi\sqrt{-1}\theta})&=ade^{2\pi\sqrt{-1}\theta}+bc\,\\
  bd(1-e^{2\pi\sqrt{-1}\varphi}) &=0 \end{aligned} \right.}.
$$ 
From this, we find that either $c=b=0$ or $a=d=0$.

{\it Subcase 1.1.} When $c=b=0$, we have $ad=1$, and $f_{2}=a^{2}f_{1}$. Consequently, the branched projective coverings defined by $f_{1}$ and $f_{2}$ lead to the same ${\mathbb P}^1$ bundle over $X$, denoted as $P$. Moreover, we can deduce that  $g_{1}$ and $g_{2}$ correspond to the pairs $(P,w)$ and $(P,a^{2}w)$, respectively.

In the context of considering extensions in the space ${\rm Ext}^1(L^{-1},\,L)^{\rm sp}$, it is necessary for the corresponding rank 2 bundle associated with $P$ to take the form of $J\oplus J^{-1}$, where $J\in{\rm Pic}^0(X)$. Assuming that $g_{1}$ and $g_{2}$ can be derived from two strictly polystable extensions, namely $L\stackrel{i} \hookrightarrow J\oplus J^{-1}$ and $L\stackrel{j} \hookrightarrow (J\oplus J^{-1})\otimes N$, respectively, where 
$J^2\neq\mathcal{O}_{X}$ and $N^2=\mathcal{O}_{X}$. Furthermore, let us establish that, up to a nonzero complex number,  the pairs $(P,w)$ and $(P,a^{2}w)$ correspond to the aforementioned extensions, $L\stackrel{i} \hookrightarrow J\oplus J^{-1}$ and $L\stackrel{j} \hookrightarrow (J\oplus J^{-1})\otimes N$, respectively. 

For the bundle $J\oplus J^{-1}$, we choose an open cover $\{U_{\alpha}\}$ of $X$ with constant transition functions on $U_{\alpha\beta}=U_{\alpha}\cap U_{\beta}$ given by
$$
\displaystyle{g_{\alpha\beta}=\begin{pmatrix} a_{\alpha\beta} & 0 \\ 0 & a_{\alpha\beta}^{-1} \end{pmatrix}},
$$
where  $a_{\alpha\beta}\in\rm{U}(1)$.
Subsequently, we conside $(e_{\alpha},e^{*}_{\alpha})$ as a holomorphic local frame for $(J\oplus J^{-1})|_{U_{\alpha}}$ with respect to $\{U_{\alpha\beta},g_{\alpha\beta}\}$. For the holomorphic subbundle $L$ of $J\oplus J^{-1}$, we choose a local holomorphic frame $l_{\alpha}$ on $U_{\alpha}$,  along with corresponding transition functions $\ell_{\alpha\beta}$. Considering the embedding $L\stackrel{i} \hookrightarrow J\oplus J^{-1}$, we find that $i(l_{\alpha})$ is a nowhere vanishing section of $J\oplus J^{-1}$ on $U_{\alpha}$. It can be expressed as $i(l_{\alpha})=s_{1,\alpha}e_{\alpha}+s_{2,\alpha}e^{*}_{\alpha}$, where $s_{1,\alpha}$ and $s_{2,\alpha}$ have no common zeros. Consequently, we can locally represent $w$ as $\left\{{U_{\alpha},w_{\alpha}=\frac{s_{1,\alpha}}{s_{2,\alpha}}}\right\}$.

In the following, we illustrate that the metric $g_{2}$ can be defined by the extension $L\stackrel{i} \rightarrow J\oplus J^{-1}\stackrel{\phi} \rightarrow J\oplus J^{-1}$, where
$$
\displaystyle{
\phi=\begin{pmatrix} a^2\mu & 0 \\ 0 & \mu \end{pmatrix}\in{\rm Aut}_{X}(J\oplus J^{-1})}.
$$ Specifically, we have
\begin{align*}
\phi\circ i(l_{\alpha})=\phi(s_{1,\alpha}e_{\alpha}+ s_{2,\alpha}e^*_{\alpha})
=s_{1,\alpha}a^2\mu e_{\alpha}+s_{2,\alpha}\mu e^*_{\alpha}.
\end{align*}
Assume that $(P,\hat{w})$ is derived from the extension $L\stackrel{\phi\circ i} \longrightarrow J\oplus J^{-1}$. Then, locally,
$\hat{w}$ can be expressed as
$$
\displaystyle{\left\{{U_{\alpha},\hat{w}_{\alpha}=\frac{a^2\mu s_{1,\alpha}}{\mu s_{2,\alpha}}=a^2w_{\alpha}}\right\}}.
$$
This implies that the two embeddings $L\stackrel{\phi\circ i} \longrightarrow J\oplus J^{-1}$ and $L\stackrel{j} \longrightarrow (J\oplus J^{-1})\otimes N$ yield the same pair $(P,a^2w)$. Furthermore, consider the embedding
$$
L\otimes N
\xrightarrow{(\phi\circ i)\otimes {\rm id}_{N}}
(J\oplus J^{-1})\otimes N,
$$
which also defines the pair $(P,a^2w)$. Therefore, we conclude that $L=L\otimes N$ and $N=\mathcal{O}_{X}$. Additionally, $j=\lambda \phi\circ i$, where $\lambda\in \mathbb{C}^*$.

{\it Subcase 1.2.} In the situation where $a=d=0$, we possess a comparable proof analogous to Subcase 1, but we omit it for brevity.\\

{\it Case 2.} Let $g_{1}$ and $g_{2}$ be two projectively equivalent trivial reducible metrics. Assume
$f_{2}=\frac{af_{1}+b}{cf_{1}+d}$, where $a, b, c, d\in\mathbb{C}$ and  $ad-bc=1$. According to \cite[Lemma 4.2.]{CWWX:2015},  $f_{1}$ and $f_{2}$ are rational functions on the Riemann surface $X$. This implies that we can construct a trivial $\mathbb{P}^{1}$-bundle from $g_{1}$ and $g_{2}$. The corresponding rank 2 bundle is $J\oplus J$ for some 2-torsion line bundle $J$. From $g_{1}$ and $g_{2}$, we obtain two pairs $(P,w)$ and $(P,\tilde{w})$, respectively. Moreover, the pairs $(P,w)$ and $(P,\tilde{w})$ correspond to the extensions $L\stackrel{i} \hookrightarrow J\oplus J$ and $L\stackrel{j} \hookrightarrow \tilde{J}\oplus \tilde{J}$ (up to a nonzero complex number), where $P=\mathbb{P}(J\oplus J)=\mathbb{P}(\tilde{J}\oplus \tilde{J})$ and $J$, $\tilde{J}$ are 2-torsion line bundles.

Below, we show that the extension $L\stackrel{i} \rightarrow J\oplus J\stackrel{\phi} \rightarrow J\oplus J$ also defines the metric $g_{2}$, where
\[
\phi=\begin{pmatrix} a & b \\ c & d \end{pmatrix}\in \operatorname{Hom}(J\oplus J,J\oplus J).
\]
Let $(e_{1,\alpha},e_{2,\alpha})$ be a holomorphic local frame for $(J\oplus J)|_{U_{\alpha}}$ with respect to $\{U_{\alpha\beta},g_{\alpha\beta}\}$. We choose a local holomorphic frame $l_{\alpha}$ of $L$ on $U_{\alpha}$ and set $i(l_{\alpha})=s_{1,\alpha}e_{1,\alpha}+ s_{2,\alpha}e_{2,\alpha}$. Then, $w$ can be locally expressed as $\left\{{U_{\alpha},w_{\alpha}=\frac{s_{1,\alpha}}{s_{2,\alpha}}}\right\}$. Assuming that we obtain the pair $(P,\hat{w})$ from the extension $L\stackrel{\phi\circ i} \longrightarrow J\oplus J$. Since $\phi(e_{1,\alpha})=ae_{1,\alpha}+ ce_{2,\alpha}$ and $\phi(e_{2,\alpha})=be_{1,\alpha}+ de_{2,\alpha}$, we have
\begin{align*}
\phi\circ i(l_{\alpha})=&\phi(s_{1,\alpha}e_{1,\alpha}+ s_{2,\alpha}e_{2,\alpha})\
=(as_{1,\alpha}+bs_{2,\alpha})e_{1,\alpha}+ (cs_{1,\alpha}+ds_{2,\alpha})e_{2,\alpha}.
\end{align*}
This implies that $\hat{w}$ has the local expression
$$
\displaystyle{\left\{{U_{\alpha},\hat{w}_{\alpha}=\frac{as_{1,\alpha}+bs_{2,\alpha}}{cs_{1,\alpha}+ds_{2,\alpha}}=\frac{aw_{\alpha}+b}{cw_{\alpha}+d}}\right\}}.
$$
Therefore, $\hat{w}=\tilde{w}$, and the extensions $L\stackrel{\phi\circ i} \longrightarrow J\oplus J$ and $L\stackrel{j} \rightarrow \tilde{J}\oplus \tilde{J}$ yield the same pair $(P,\tilde{w})$.

Consider the extension
$$
L\otimes (J\otimes \tilde{J})
\xrightarrow{(\phi\circ i)\otimes {\rm Id}_{J\otimes \tilde{J}}}
(J\oplus J)\otimes (J\otimes \tilde{J})=\tilde{J}\oplus \tilde{J},
$$
which also yields the pair $(P,\tilde{w})$. Consequently, we have $L=L\otimes (J\otimes \tilde{J})$, $J\otimes \tilde{J}=\mathcal{O}_{X}$, and $J=\tilde{J}$. Thus, the extensions $L\stackrel{\phi\circ i} \longrightarrow J\oplus J$ and $L\stackrel{j} \rightarrow J\oplus J$ are co-linear in ${\rm Ext}^1(L^{-1},\,L)$.

\end{proof}

\begin{rem}
In the proof of Proposition \ref{prop:red}, we observe that when $J$ is a 2-torsion, the resulting $\mathbb{P}^{1}$-bundle $P$ is identified as the trivial $\mathbb{P}^{1}$ bundle. This leads to the acquisition of a trivial reducible metric. A family of projectively equivalent trivial reducible metrics can be defined by three real parameters, as documented by Umehara--Yamada (2000) (\cite[Fact 2.1., p.77]{Umehara-Y:2000}).

On the other hand, in cases where $J$ is not a 2-torsion, the process yields a nontrivial reducible metric. Correspondingly, a family of projectively equivalent nontrivial reducible metrics can be defined by one real parameters (\cite[Fact 2.1., p.77]{Umehara-Y:2000}).
\end{rem}

\section{Effective divisors represented by reducible metrics }
This section primarily focuses on introducing the concept of the restricted ramification divisor map and subsequently estimating the rank of its complex derivatives from below. 
Utilizing this estimate, we proceed to present the proofs for Theorems \ref{thm:gen} and \ref{thm:ell} in Subsection 5.3. In this context, Corollary \ref{cor:generic} is established through a dimensional analysis approach.

\subsection{Strictly polystable short exact sequence}
Given an embedding $L\stackrel{i}{\to} J\oplus J^{-1}$, we can canonically construct a holomorphic map $J\oplus J^{-1}\stackrel{p}{\to} L^{-1}$ such that the following short exact sequence of holomorphic vector bundles is formed:
\[
0\to L\stackrel{i}{\to} J\oplus J^{-1}\stackrel{p}{\to} L^{-1}\to 0.
\]
 This construction will be used in the next subsection.
\begin{lem}
\label{lem:construction of p}
Let $L,\,\,J$ be two holomorphic line bundles over a compact Riemann surface $X$ with $\deg\,L\leq 0$ and $\deg\,J= 0$. Suppose that we have an embedding  $L\stackrel{i} \hookrightarrow J\oplus J^{-1}$. Then we can construct a nowhere-vanishing holomorphic section $p\colon J\oplus J^{-1}\rightarrow L$ of ${\rm Hom}(J\oplus J^{-1},\, L)$ in a canonical way. This construction yields the following short exact sequence of holomorphic vector bundles:
\[
0\to L\stackrel{i}{\to} J\oplus J^{-1}\stackrel{p}{\to} L^{-1}\to 0.
\]
\end{lem}

\begin{proof}
Let us consider the canonical isomorphic map 
\begin{equation}
\label{equ:varphi 1}
\varphi:{\rm Hom}(L,J) \rightarrow {\rm Hom}(J^{-1},L^{-1}),
\end{equation}
which relates a local section $s_{1}$ of ${\rm Hom}(L,J)$ to a local section $v^{*}$ of $J^{-1}$ through the equation  $\varphi(s_{1})(v^{*}):=v^{*}\circ s_{1}$.  Here, $s_{1}$ represents a local section of ${\rm Hom}(L,J)$, and $v^{*}$ represents a local section of $J^{-1}$.
Similarly,  we define another canonical isomorphism map 
\begin{equation}
\label{equ:varphi 2}
\varphi^{\vee}:{\rm Hom}(L,J^{-1})\rightarrow {\rm Hom}(J,L^{-1}),
\end{equation}
which relates a local section $s_{2}$ of ${\rm Hom}(L,J^{-1})$ to a local section $u$ of $J$ through the equation  $\varphi^{\vee}(s_{2})(u):=\rm{ev}_{u}\circ s_{2}$. In this case, $s_{2}$ represents a local section of ${\rm Hom}(L,J^{-1})$, $u$ represents a local section of $J$, and ${\rm ev}:J\rightarrow (J^{-1})^{-1}$ is the canonical isomorphic map given by ${\rm ev}_{u}(v^{*}):=v^{*}(u)$.

The embedding $L\stackrel{i} \hookrightarrow J\oplus J^{-1}$ can be represented by a pair $(i_1,\,i_2)$ of two sections $i_1$ and $i_2$ belonging to
$H^0\big({\rm Hom}(L,\,J)\big)$ and $H^0\big({\rm Hom}(L,\,J^{-1})\big)$, respectively, with no common zero. Since $\varphi$ and $\varphi^{\vee}$ are holomorphic isomorphisms,  then $\varphi(i_{1})$ and $\varphi^{\vee}(i_{2})$ are two sections in ${\rm Hom}(J^{-1},L^{-1})$ and ${\rm Hom}(J,L^{-1})$, respectively, with no common zero.
Denote by $p_{1}=\varphi^{\vee}(i_{2})$ and  $p_{2}=-\varphi(i_{1})$. Then we define the map as follows:
\begin{align*}
\Upsilon\colon{\rm Hom}(J,L^{-1})\oplus {\rm Hom}(J^{-1},L^{-1})&\rightarrow {\rm Hom}(J\oplus J^{-1},L^{-1})\\
(p_{1},p_{2}) &\longmapsto \Big((u\oplus v^{*})\mapsto p_{1}(u)+p_{2}(v^{*})\Big)
\end{align*}
where $u$ is a local section of $J$, $v^*$ is a local section of $J^{-1}$, and $p:=\Upsilon(p_{1},p_{2})$.

We claim that $\textup{Im}\ i= \textup{Ker}\ p$. Actually,  we choose an open covers $\{U_{\alpha}\}$ for $X$ and constant transition functions
$$g_{\alpha\beta}=\begin{pmatrix} a_{\alpha\beta} & 0 \\ 0 & a_{\alpha\beta}^{-1} \end{pmatrix}$$
on $U_{\alpha\beta}=U_{\alpha}\cap U_{\beta}$,
where  $a_{\alpha\beta}\in\rm{U}(1)$. Let $(e_{\alpha},e^{*}_{\alpha})$ be a holomorphic flat frame for $(J\oplus J^{-1})|_{U_{\alpha}}$ with respect to $\{U_{\alpha\beta},g_{\alpha\beta}\}$.  We also choose a local holomorphic frame $l_{\alpha}$ of $L$ on $U_{\alpha}$, along with the corresponding transition functions $\ell_{\alpha\beta}$ such that $l_{\beta}=\ell_{\alpha\beta}l_{\alpha}$. For the embedding $L\stackrel{i} \hookrightarrow J\oplus J^{-1}$, we set $i(l_{\alpha})=i_{1}(l_{\alpha})+i_{2}(l_{\alpha})=s_{1,\alpha}e_{\alpha}+s_{2,\alpha}e^{*}_{\alpha}$. By performing the following calculations:
\begin{align*}
p\circ i(l_{\alpha})=&p\circ(s_{1,\alpha}e_{\alpha}+s_{2,\alpha}e^{*}_{\alpha})\\
=&s_{1,\alpha}\varphi^{\vee}(i_{2})(e_{\alpha})-s_{2,\alpha}\varphi(i_{1})(e_{\alpha}^*)=0,
\end{align*}
we find that $p\circ i=0$, i.e. $p_{1}\circ i_{1}+p_{2}\circ i_{2}=0$. Therefore, we have shown that $\textup{Im} i\subseteq \textup{Ker}\ p$.
It should be noted that $p(e_{\alpha})=p_{1}(e_{\alpha})=\varphi^{\vee}(i_{2})(e_{\alpha})\neq0$, implying that $p(x)\neq0$ whenever $i_{2}(x)\neq0$, where $x\in X$. Similarly, $p(x)\neq0$ for any $i_{1}(x)\neq0$. Since $i_1$ and $i_2$ do not have any common zeros, it follows that $p$ is a nowhere-vanishing section of ${\rm Hom}(J\oplus J^{-1},L^{-1})$. Furthermore, for every $x\in X$, the dimension of $\textup{Im}\ i|_{x}$ is equal to $1$, which implies that $\textup{Im}\ i= \textup{Ker}\ p$. Consequently, we have obtained a short exact sequence of holomorphic vector bundles
\[
0\to L\stackrel{i}{\to} J\oplus J^{-1}\stackrel{p}{\to} L^{-1}\to 0.
\]

\end{proof}

\subsection{Restricted ramification divisor map}
In this subsection, we delve deeper into the study of reducible metrics by introducing and analyzing the \textit{restricted ramification divisor map}. We will compute its complex differential and estimate its rank, emphasizing the implications of this concept. Originating from a specific restriction of the ramification divisor map discussed in Subsection 2.1, this map is pivotal in identifying and characterizing effective divisors of even degrees that are associated with reducible metrics.

\begin{prop}
\label{prop:restricted-rdm}
Let $L,\,\,J$ be two holomorphic line bundles over a compact Riemann surface $X$ of genus $g_X>0$ such that

\begin{itemize}
    \item $\deg J=0$; 

\item $\dim H^0\big(\operatorname{Hom}(L,J)\big)>0$ \text{and} $\dim H^0\big(\operatorname{Hom}(L,J^*)\big)>0$.
\end{itemize}
Utilizing the two canonical isomorphisms $\varphi$ and $\varphi^\vee$ from (\ref{equ:varphi 1}, \ref{equ:varphi 2}), we define the {\rm restricted ramification divisor map} as follows:
\begin{align*}
\rho\colon H^0\big(X, \mathrm{Hom}(L, J)\big) \oplus H^0\big(X, \mathrm{Hom}(L, J^{-1})\big) &\rightarrow H^0\big(X, \mathrm{Hom}(L, L^{-1} \otimes K_X)\big)\\
(i_0, \widehat{i_0}) &\mapsto \varphi^{\vee}(\widehat{i_0}) \circ \partial_{J} \circ i_{0} - \varphi(i_{0}) \circ \partial_{J^{-1}} \circ \widehat{i_0},
\end{align*}
where $\partial_{J}$ and $\partial_{J^{-1}}$ are the $(1,0)$-parts of the Hermitian-Einstein connections $D_{J}$ on $J$ and $D_{J^{-1}}$ on $J^{-1}$, respectively. The following two statements hold{\rm :}
\begin{enumerate}
    \item $\rho$ is a bilinear holomorphic map, and its image constitutes a cone within the space  $H^0\big(X, \mathrm{Hom}(L, L^{-1} \otimes K_X)\big)$. 
    The complex derivative of $\rho$ at $(i_{1},i_{2}) \in H^0\big(\mathrm{Hom}(L, J)\big) \oplus H^0\big(\mathrm{Hom}(L, J^{-1})\big)$ along $(i_{0},\widehat{i_0})$ is given by
    \begin{align*}
    \rho_{*(i_{1},i_{2})}(i_{0},\widehat{i_0}) = &\ \varphi^{\vee}(i_{2}) \circ \partial_{J} \circ i_{0} + \varphi^{\vee}(\widehat{i_0}) \circ \partial_{J} \circ i_{1} \\
    &\ - \varphi(i_{1}) \circ \partial_{J^{-1}} \circ \widehat{i_0} - \varphi(i_{0}) \circ \partial_{J^{-1}} \circ i_{2}.
    \end{align*}
    Particularly, if $J^{-1} = J$, i.e., $J^2 = \mathcal{O}_X$, then $\varphi = \varphi^\vee$, $\rho$ is anti-symmetric on $H^0\big(\mathrm{Hom}(L, J)\big) \oplus H^0\big(\mathrm{Hom}(L, J^{-1})\big)$, and
    \begin{align*}
    \rho_{*(i_{1},i_{2})}(i_{0},\widehat{i_0}) = &\ \varphi(i_{2}) \circ \partial_{J} \circ i_{0} + \varphi(\widehat{i_0}) \circ \partial_{J} \circ i_{1} \\
    &\ - \varphi(i_{1}) \circ \partial_{J} \circ \widehat{i_0} - \varphi(i_{0}) \circ \partial_{J} \circ i_{2}.
    \end{align*}
    
    \item Suppose $i := (i_{1},i_{2}) \in H^0\big(\mathrm{Hom}(L, J)\big) \oplus H^0\big(\mathrm{Hom}(L, J^{-1})\big)$ embeds $L$ into $J \oplus J^{-1}$, yielding a short exact sequence $0 \to L \stackrel{i}{\hookrightarrow} J \oplus J^{-1} \stackrel{p}{\rightarrow} L^{-1} \to 0$ as in Lemma \ref{lem:construction of p}. Then $\rho(i_1, i_2)$ coincides with the second fundamental form $\beta^*$ of the embedding $L \stackrel{i}{\hookrightarrow} J \oplus J^{-1}$, where $J \oplus J^{-1}$ is endowed with a Hermitian-Einstein metric and $L$ with the induced metric.
\end{enumerate}
\end{prop}

\begin{proof}
(1) We first note that the expression for $\rho\big(i_0, \widehat{i_0}\big)$ can be articulated as follows:
\begin{equation}
\label{equ:rho_2nd}
\rho\big(i_0, \widehat{i_0}\big) = \varphi^{\vee}(\widehat{i_0}) \circ \partial_{J} \circ i_{0} - \varphi(i_{0}) \circ \partial_{J^{-1}} \circ \widehat{i_0} = (p \otimes \mathrm{id}) \circ \partial_{J \oplus J^{-1}} \circ i.
\end{equation}
Here, $i$ is the holomorphic section of $\mathrm{Hom}(L, J \oplus J^{-1})$, determined by $(i_{0},\widehat{i_0})$, and $p$ is as constructed in the proof of Lemma \ref{lem:construction of p} from $i$. Note that $\partial_{J \oplus J^{-1}}$ denotes the $(1,0)$-part of the Hermitian-Einstein connection $D_{J \oplus J^{-1}}$. Thus, $\rho(i_{0},\widehat{i_0})$ is a holomorphic section of $\mathrm{Hom}(L, L^{-1} \otimes K_X)$. Moreover, for any $\lambda, \mu \in \mathbb{C}$, it holds that $\rho(\lambda i_{0}, \mu \widehat{i_0}) = \lambda \mu \rho(i_{0}, \widehat{i_0})$. Upon substituting $i_{t} := (i_{1} + ti_{0}, i_{2} + t\widehat{i_0})$, for $t \in \mathbb{C}$, into $\rho$'s definition, we deduce:
\begin{align*}
\rho(i_t) &= (p_{t} \otimes \mathrm{id}) \circ \partial_{J \oplus J^{-1}} \circ i_{t} \\
&= (p_{t} \otimes \mathrm{id}) \circ (\partial_{J} \oplus \partial_{J^{-1}}) \circ (i_{1} + ti_{0}, i_{2} + t\widehat{i_0}) \\
&= [\varphi^{\vee}(i_{2}) + t\varphi^{\vee}(\widehat{i_0})] \circ \partial_{J} \circ (i_{1} + ti_{0}) - [\varphi(i_{1}) + t\varphi(i_{0})] \circ \partial_{J^{-1}} \circ (i_{2} + t\widehat{i_0}).
\end{align*}
The derivative can then be calculated as:
\begin{equation}
\label{equ:derivative}
\begin{aligned}
\frac{\mathrm{d}}{{\mathrm{d}}t}\bigg|_{t=0} \rho(i_t) = \varphi^{\vee}(i_{2}) \circ \partial_{J} \circ i_{0} + \varphi^{\vee}(\widehat{i_0}) \circ \partial_{J} \circ i_{1} - \varphi(i_{1}) \circ \partial_{J^{-1}} \circ \widehat{i_0} - \varphi(i_{0}) \circ \partial_{J^{-1}} \circ i_{2}.
\end{aligned}
\end{equation}

(2) The second statement follows directly from Equations \eqref{equ:rho_2nd}, \eqref{equ:2ndfund}, and Lemma 
\ref{lem:well_defn_1}.
\end{proof}

Through the application of local flat trivialization of $J\oplus J^{-1}$, we can estimate from below the rank for the complex differential of the 
restricted ramification divisor map. This lemma plays a pivotal role in the subsequent derivation of Theorems \ref{thm:gen} and \ref{thm:ell} in the following subsection.

\begin{lem}
\label{lemma:lower bound}
Let $L$ and $J$ be two holomorphic line bundles over a compact Riemann surface $X$ of genus $g_X > 0$, where $\deg J = 0$ and $-k := \deg L \leq 1 - 2g_X$. Denote the complex dimension of the spaces $H^0\big(\mathrm{Hom}(L, J^{\pm 1})\big)$ by $d$. We consider two nonzero elements $i_{1} \in H^0\big(\mathrm{Hom}(L, J)\big)$ and $i_{2} \in H^0\big(\mathrm{Hom}(L, J^{-1})\big)$. An estimation for the rank of the complex differential $\rho_{*}$ of the restricted ramification divisor map $\rho\colon H^0\big(X, \mathrm{Hom}(L, J)\big) \oplus H^0\big(X, \mathrm{Hom}(L, J^{-1})\big) \rightarrow H^0\big(X, \mathrm{Hom}(L, L^{-1} \otimes K_X)\big)$ at $(i_{1},i_{2})$ can be established as follows{\rm :}
\begin{itemize}
    \item If $J^2 \neq \mathcal{O}_X$, then $d \leq \operatorname{rank} \rho_{*(i_{1},i_{2})} \leq 2d - 1$.
    In particular, we deduce that $\operatorname{rank} \rho_{*(i_{1},i_{2})} = 1$ when $d = 1$.
    \item If $J^2 = \mathcal{O}_X$ and $i_{1}, i_{2}$ are not co-linear in $H^0\big(X, \mathrm{Hom}(L, J)\big) $, then $d - 1 \leq \operatorname{rank} \rho_{*(i_{1},i_{2})} \leq 2d - 3$. 
    In particular, we deduce that $\operatorname{rank} \rho_{*(i_{1},i_{2})} = 1$ when $d = 2$.
\end{itemize}
\end{lem}

\begin{proof}
Applying the Riemann-Roch theorem, we deduce that the spaces $H^0\big(\mathrm{Hom}(L, J^{\pm 1})\big)$ possess a complex dimension of $d=k + 1 - g_X \geq g_X >0$.
We choose an open cover $\{U_{\alpha}\}$ of $X$, such that $J$ possesses a flat trivializing coordinate. Assuming that the constant transition functions of $J\oplus J^{-1}$ on $U_{\alpha\beta}=U_{\alpha}\cap U_{\beta}$ are given by
$$g_{\alpha\beta}=\begin{pmatrix} a_{\alpha\beta} & 0 \\ 0 & a_{\alpha\beta}^{-1} \end{pmatrix},$$
where  $a_{\alpha\beta}\in\rm{U}(1)$.
Let $(e_{\alpha},e^{*}_{\alpha})$ represent a holomorphic flat frame of $(J\oplus J^{-1})|_{U_{\alpha}}$ with respect to $\{U_{\alpha\beta},g_{\alpha\beta}\}$. In other words, we have
$$\begin{pmatrix}
  {e}_{\beta} \\
  {e}^{*}_{\beta}
\end{pmatrix} = \begin{pmatrix}
  a_{\alpha\beta} & 0 \\
  0 & a^{-1}_{\alpha\beta}
\end{pmatrix} \begin{pmatrix}
 {e}_{\alpha} \\
  {e}^{*}_{\alpha}
\end{pmatrix}.$$  
Choose a holomorphic frame $l_{\alpha}$ of $L$ on $U_{\alpha}$, with transition functions $\ell_{\alpha\beta}$ such that $l_{\beta}=\ell_{\alpha\beta}l_{\alpha}$.  Denote $i(l_{\alpha})=i_{1}(l_{\alpha})+i_{2}(l_{\alpha})=s_{1,\alpha}e_{\alpha}+s_{2,\alpha}e^{*}_{\alpha}$ and $i_{0}(l_{\alpha})=u_{\alpha}e_{\alpha}$ and $\widehat{i_0}(l_{\alpha})=v_{\alpha}e_{\alpha}^{*}$. Then we have
\begin{align*}
i_{t}(l_{\alpha})=(i_{1}+t\ i_{0})(l_{\alpha})+(i_{2}+t\ \widehat{i_0})(l_{\alpha})
=(s_{1,\alpha}+t\ u_{\alpha})e_{\alpha}+(s_{2,\alpha}+t\ v_{\alpha})e^{*}_{\alpha}.
\end{align*}
where $i_t:=(i_1+t i_0, i_2+t\hat{i_0})$.
Moreover, we get $s_{1,\alpha}=\ell^{-1}_{\alpha\beta}a_{\alpha\beta}\cdot s_{1,\beta}$, $s_{2,\alpha}=\ell^{-1}_{\alpha\beta}a^{-1}_{\alpha\beta}\cdot s_{2,\beta}$, $u_{\alpha}=\ell^{-1}_{\alpha\beta}a_{\alpha\beta}\cdot u_{\beta}$ and $v_{\alpha}=\ell^{-1}_{\alpha\beta}a^{-1}_{\alpha\beta}\cdot v_{\beta}$ on $U_{\alpha}\cap U_{\beta}$.
Since $\rho(i_t)(l_{\alpha})$ is a local section of $L^{-1}\otimes K_{X}$ on $U_{\alpha}$, $\rho(i_t)(l_{\alpha})$ can acts on $l_{\alpha}$.  Then we obtain the local expression of $\rho_{*(i_{1},i_{2})}\big(i_{0},\widehat{i_0}\big)$  by using the expression in (\ref{equ:derivative}):
\begin{align}
\label{equation:tangent map}
\frac{\rm d}{{\rm d}t}\bigg|_{t=0}\, \rho(i_t)(l_{\alpha})(l_{\alpha})&={\rm d}u_{\alpha}\cdot s_{2,\alpha}+{\rm d}s_{1,\alpha}\cdot v_{\alpha}-{\rm d}v_{\alpha}\cdot s_{1,\alpha}-{\rm d}s_{2,\alpha}\cdot u_{\alpha}.\nonumber \\
&=\left(\frac{{\rm d}u_{\alpha}}{{\rm d}z_{\alpha}}\cdot s_{2,\alpha}+\frac{{\rm d}s_{1,\alpha}}{{\rm d}z_{\alpha}}\cdot v_{\alpha}-\frac{{\rm d}v_{\alpha}}{{\rm d}z_{\alpha}}\cdot s_{1,\alpha}-\frac{{\rm d} s_{2,\alpha}}{{\rm d}z_{\alpha}}\cdot u_{\alpha}\right){\rm d}z_{\alpha}
\end{align}
The calculation gives us the following result on $U_{\alpha}\cap U_{\beta}$:
\[
\frac{{\rm d}u_{\alpha}}{{\rm d}z_{\alpha}}\cdot s_{2,\alpha}-\frac{{\rm d}s_{2,\alpha}}{{\rm d}z_{\alpha}}\cdot u_{\alpha}=\ell_{\alpha\beta}^{-2}\cdot\frac{{\rm d}z_{\beta}}{{\rm d}z_{\alpha}}\left(\frac{{\rm d}u_{\beta}}{{\rm d}z_{\beta}}\cdot s_{2,\beta}-\frac{{\rm d}s_{2,\beta}}{{\rm d}z_{\beta}}\cdot u_{\beta}\right)
\]
and
\[
\frac{{\rm d}s_{1,\alpha}}{{\rm d}z_{\alpha}}\cdot v_{\alpha}-\frac{{\rm d}v_{\alpha}}{{\rm d}z_{\alpha}}\cdot s_{1,\alpha}=\ell_{\alpha\beta}^{-2}\cdot\frac{{\rm d}z_{\beta}}{{\rm d}z_{\alpha}}\left(\frac{s_{1,\beta}}{{\rm d}z_{\beta}}\cdot v_{\beta}-\frac{{\rm d}v_{\beta}}{{\rm d}z_{\beta}}\cdot s_{1,\beta}\right),
\]
Thus,  $\left\{U_{\alpha},\frac{{\rm d}u_{\alpha}}{{\rm d}z_{\alpha}}\cdot s_{2,\alpha}-\frac{{\rm d}s_{2,\alpha}}{{\rm d}z_{\alpha}}\cdot u_{\alpha}\right\}$ defines the holomorphic section $\varphi^{\vee}(i_{2})\circ \partial_{J}\circ i_{0}-\varphi(i_{0})\circ \partial_{J^{-1}}\circ i_{2}$. Similarly, 
$\left\{U_{\alpha},\frac{{\rm d}v_{\alpha}}{{\rm d}z_{\alpha}}\cdot s_{1,\alpha}-\frac{{\rm d}s_{1,\alpha}}{{\rm d}z_{\alpha}}\cdot v_{\alpha}\right\}$ 
corresponds to 
the holomorphic section $\varphi(i_{1})\circ \partial_{J^{-1}}\circ \hat{i_{0}}-\varphi^{\vee}(\hat{i_{0}})\circ \partial_{J}\circ i_{1}$.

Moreover, when $(i_0, {\hat {i_0}})=( i_{1},- i_{2})$, we can easily see from Proposition \ref{prop:restricted-rdm} that the above derivative vanishes. Therefore, we can conclude that
${\rm rank}\ \rho_{*(i_{1},i_{2})}\leq 2d-1$.

Let us consider the following case where we set $\hat{i_{0}}=0$, implying $v_{\alpha}=0$. In other words, we are considering
$i_{t}=\big(i_{1}+t\ i_{0},i_{2}\big)$, $t\in \mathbb{C}$.
In this scenario, we have
\[
\frac{\rm d}{{\rm d}t}\bigg|_{t=0}\, \rho(i_t)(l_{\alpha})(l_{\alpha})={\rm d}u_{\alpha}\cdot s_{2,\alpha}-{\rm d}s_{2,\alpha}\cdot u_{\alpha}.
\]
Consequently, if $\frac{\rm d}{{\rm d}t}\bigg|_{t=0}\, \rho(i_t)=0$, then ${\rm d}u_{\alpha}\cdot s_{2,\alpha}-{\rm d}s_{2,\alpha}\cdot u_{\alpha}=0$ and $u_{\alpha}=C_{\alpha}\cdot s_{2,\alpha}$, where $C_{\alpha}$ is a complex constant on $U_{\alpha}$.

{\it Case 1.} If $C_{\alpha}=0$, it implies that $i_{0}=0$ and the restriction of the complex differential $\rho_{*(i_1,i_2)}$ to
the subspace $H^0\big(X,\, {\rm Hom}(L,\,J)\big)\oplus \{0\}$ of 
 $H^0\big(X,\, {\rm Hom}(L,\,J)\big)\oplus H^0\big(X,\, {\rm Hom}(L,\,J^{-1})\big) $
is non-degenerate. Thus, we can conclude that  ${\rm rank}\ \rho_{*(i_{1},i_{2})}\geq d$.

{\it Case 2.} Suppose $C_{\alpha}\neq0$. Considering
\[
0=\frac{\rm d}{{\rm d}t}\bigg|_{t=0}\, \rho(i_t)(l_{\beta})(l_{\beta})={\rm d}u_{\beta}\cdot s_{2,\beta}-{\rm d}s_{2,\beta}\cdot u_{\beta}\quad {\rm on}\quad U_\beta,
\]
and recalling $u_{\alpha}=\ell^{-1}_{\alpha\beta}a_{\alpha\beta}\cdot u_{\beta}$ on $U_{\alpha}\cap U_{\beta}$, we can deduce that $u_{\beta}\neq0$ since $i_2\not=0$. Since $u_{\beta}=C_{\beta}\cdot s_{2,\beta}$, where $C_{\beta}$ is a non-zero complex constant on $U_{\beta}$, we find that $C_{\alpha}=a_{
\alpha\beta}^2\cdot C_{\beta}$. Moreover, the pair $\{U_{\alpha},C_{\alpha}\}$ defines a global holomorphic section of $J^2$, leading us to $J^2=\mathcal{O}_{X}$. Therefore, we have $a^{2}_{\alpha\beta}=1$ and $C_{\alpha}=C_{\beta}:=C$. Actually, we obtain
$i_{t}=\big(i_{1}+t\ C\cdot i_{2},i_{2}\big)$, $t\in \mathbb{C}$,
which define a family of projectively equivalent trivial reducible metrics on $X$.
Furthermore, in this case, ${\rm Ker}\,\rho_{*(i_{1},i_{2})}$ contains at least three linearly independent vectors $(i_{2},0)$, $(0,i_{1})$ and $(i_{1},-i_{2})$ based on either Proposition \ref{prop:restricted-rdm} or the local expression (\ref{equation:tangent map}) of $\rho_{*(i_1,i_2)}$. Consequently, we can conclude that $d-1\leq\operatorname{rank} \rho_{*(i_{1},i_{2})}\leq 2d-3$.

\end{proof}

\subsection{Existence and non-uniqueness of integral reducible metrics}

\begin{proof}[Proof of Theorem \ref{thm:gen}]

Let $X$ be a compact Riemann surface with genus $g_X\geq 2$ and suppose that $\deg\, D\geq 6g_X-2$.
Since $\deg\big( K_X^{-1}\otimes {\mathcal O}_X(D)\big)\geq 4g_X$ and $K_X^{-1}\otimes {\mathcal O}_X(D)$ has even degree, we can choose a line bundle $L$ with $\deg L \leq -2g_X$ such that $L^{-2}= K_X^{-1}\otimes {\mathcal O}_X(D)$. Then there exists a strictly polystable extension in ${\rm Ext}^1(L^{-1},\, L)$ by Lemma \ref{lemma:ample}.  Additionally, according to Proposition \ref{prop:red}, the projective class of this strictly polystable extension in ${\mathbb P}\bigl({\rm Ext}^1(L^{-1},\,L)^{\rm sp}\bigr)$ gives rise to a family of reducible metrics that represent some divisor $D'\in |D|$.

Let $L\stackrel{i} \hookrightarrow J\oplus J^{-1}$ be the embedding obtained as mentioned above. This can be equivalently represented by a pair $(i_1,\,i_2)$ of two sections in
$H^0\big({\rm Hom}(L,\,J)\big)$ and $H^0\big({\rm Hom}(L,\,J^{-1})\big)$, respectively, where these sections have no common zero. Now, let's take $i_{0}$ and $\widehat{i_0}$ as any holomorphic sections of ${\rm Hom}(L,\,J)$ and ${\rm Hom}(L,\,J^{-1})$ respectively. We denote
\[
i_{t}=\big(i_{1}+t\ i_{0},i_{2}+t\ \widehat{i_0}\big) \quad t\in \mathbb{C},
\]
and for sufficiently small $t$, $i_{t}$ forms embeddings $L\to J\oplus J^{-1}$.  This implies that $\rho(i_t)$, as defined in Proposition \ref{prop:restricted-rdm}, corresponds exactly to the second fundamental forms $\beta^{*_{t}}$ 
for $|t|<<1$.  According to Lemma \ref{lemma:divisor}, we know that the divisors of $\beta^{*_{t}}$ as holomorphic sections of  ${\rm Hom}(L,\,L^{-1}\otimes K)$ are precisely represented by the reducible metrics defined by the embeddings $L \stackrel{i_{t}} \hookrightarrow J\oplus J^{-1}$ for $|t|<<1$. 
Therefore, in combination with Lemma \ref{lemma:lower bound}, we obtain the following conclusion:

\begin{itemize}
\item If $J^{2}\neq{\mathcal O}_X$, then there exists an arcwise connected Borel subset of $|D|$ with Hausdorff dimension $2(d-1)=2(-\deg\, L-g_{X})=\deg\, D-4g_{X}+2$. This subset contains $D'$, and each divisor in it can be represented by nontrivial reducible metrics.

\item Suppose that $J^{2}={\mathcal O}_X$. Since $i_{1}\neq \lambda i_{2}$ for any $\lambda\in\mathbb{C}$, there exists an arcwise connected Borel subset of $|D|$ with Hausdorff dimension $2(d-1-1)=2(-\deg\, L-g_{X}-1)=\deg\, D-4g_{X}$. This subset contains $D'$, and each divisor in it can be represented by trivial reducible metrics.
\end{itemize}

\end{proof}

Based on the proof of Theorem \ref{thm:gen}, when considering the scenario where $g_{X}=1$,  we can derive analogous results as presented in Theorem \ref{thm:ell}. 

\begin{proof}[Proof of Theorem \ref{thm:ell}]

Since $\deg\big( K_X^{-1}\otimes {\mathcal O}_X(D)\big)\geq 2$ and $K_X^{-1}\otimes {\mathcal O}_X(D)$ has even degree, we can choose a line bundle $L$ with $\deg L \leq -1$ such that $L^{-2}= K_X^{-1}\otimes {\mathcal O}_X(D)$.  Consequently, an embedding of $L\rightarrow J\oplus J^{-1}$ exists, where $J$ is a line bundle of degree zero, as established by Theorem \ref{thm:inclusionindirectsum}. According to Proposition \ref{prop:red}, the projective class of this embedding $L\to J\oplus J^{-1}$ in ${\mathbb P}\bigl({\rm Ext}^1(L^{-1},\,L)^{\rm sp}\bigr)$ yields a family of reducible metrics representing a certain divisor $D'\in |D|$. 
Moreover, in the scenario where $\deg D\geq 4$, in conjunction with Theorem \ref{thm:gen}, we deduce Theorem \ref{thm:ell}. In the instance where $\deg D=2$, the proof is left to Subsection 5.4.

\end{proof}

\begin{proof}[Proof of Corollary \ref{cor:generic}]
This proof concentrates on the case where $g_X \geq 2$, as the methodology for $g_X = 1$ is analogous. Our approach is based on a dimensional counting argument. For clarity, let's denote the degree of $D$ as $\mathfrak d$. Recall that the symmetric product ${\rm Sym}^{\mathfrak d}(X)$ represents the moduli space of effective divisors of degree $\mathfrak d$ on $X$. We define $\mathcal{R}_{\mathfrak d}$ as the subspace comprising ramification divisors of branched covers $X \to \mathbb{P}^1$ with a degree of $\frac{{\mathfrak d} + 2 - 2g_X}{2}$. Notably, effective divisors of degree $\mathfrak d$ that are represented by trivial reducible metrics constitute the space $\mathcal{R}_{\mathfrak d}$, as per \cite[Theorem 3.4]{CWWX:2015}.
Our goal is to estimate the complex dimension of $\mathcal{R}_{\mathfrak d}$ from above. The focus will be on the subset of simple ramification divisors within $\mathcal{R}_{\mathfrak d}$. According to the Riemann existence theorem (refer to \cite[Section 4.2.]{Donaldson:2011}), this subset's complex dimension does not exceed $({\mathfrak d} - 3)$. This implies that a generic element in ${\rm Sym}^{\mathfrak d}(X)$ cannot be a ramification divisor of any branched cover $X \to \mathbb{P}^1$.
\end{proof}

\subsection{cone metrics representing effective divisors with degree 2 on an elliptic curve}
In this subsection, we provide the proof of Theorem \ref{thm:ell} for the case where $\deg D = 2$. Before proceeding, we need to establish the following proposition:
\begin{prop}
\label{prop:extensionsoflinebundle-1}
Let $X$ represent an elliptic curve over $\mathbb{C}$, and let $L$ be a line bundle of degree $-1$ over $X$. Then, in the space $\operatorname{Ext}_{X}^{1}(L^{-1}, L) \cong \mathbb{C}^{2}$, besides the strictly polystable extensions given by
\[ 0 \to L \to J \oplus J^{-1} \to L^{-1} \to 0, \quad J \in \operatorname{Pic}^{0}(X) \setminus \{ 2-torsions\}, \]
there exist precisely four complex lines consisting of extensions of the form
\[ 0 \to L \to E \to L^{-1} \to 0, \quad E \in  \mathcal{E}(2,0), \]
where $\mathcal{E}(2,0)$ denote the set of isomorphism classes of indecomposable vector bundles of rank 2 and degree 0 over $X$.
Consequently, $\mathbb{P}\Big({\rm Ext}_{X}^{1}(L^{-1},L)^{\rm sp}\Big)\cong \mathbb{P}^{1}\setminus \{4\  points\}$ forms a path-connected space, where ${\mathbb P}\bigl({\rm Ext}^1(L^{-1},L)^{\rm sp}\bigr)$ denotes the projectivization of the space ${\rm Ext}^1(L^{-1},L)^{\rm sp}$ containing strictly polystable extensions of $L^{-1}$ by $L$.

\end{prop}

\subsubsection{Proof of Proposition \ref{prop:extensionsoflinebundle-1}}

Firstly, in accordance with \cite[Theorem 5]{Atiyah:1957elliptic}, a unique vector bundle $E_{0} \in \mathcal{E}(2,0)$ exists, such that $E_{0}$ fits into an exact sequence:
\[ 0 \to \mathcal{O}_{X} \to E_{0} \to \mathcal{O}_{X} \to 0. \]
Subsequently, we establish the following lemma:

\begin{lem}
\label{lem:subbundleofE0}
Suppose $L_{0} \in \operatorname{Pic}^{0}(X)$ is a line subbundle of $E_{0}$. Then $L_{0} \cong \mathcal{O}_{X}$.
\end{lem}
\begin{proof}
Consider the following diagram:
\[ \xymatrix{
 & & L_{0}\ar[d]^{j} \ar[rd]^{p} \ar@{-->}[ld]_{i} & & \\
 0 \ar[r] & \mathcal{O}_{X}\ar[r]^{i_{0}} & E_{0} \ar[r]^{p_{0}} & \mathcal{O}_{X} \ar[r] & 0
}\] 
If $p = p_{0} \circ j \colon L_{0} \to \mathcal{O}_{X}$ is nonzero, then $p$ must be an isomorphism because $\deg L_{0} = 0 = \deg \mathcal{O}_{X}$.
If $p = 0$, then there exists an inclusion $i \colon L_{0} \to \mathcal{O}_{X}$, which is also an isomorphism.
\end{proof}

\begin{rem}
\label{rem:subbundleofE0}
If $S$ is a line bundle of degree greater than zero, it cannot be a line subbundle of $E_{0}$. Otherwise, we would have the extension $0 \to S \to E_{0} \to S^{-1} \to 0$, which would be splitting, as $\operatorname{Ext}_{X}^{1}(S^{-1}, S) \cong H^{1}(X, S^{2}) = 0$.
\end{rem}

\begin{lem}
\label{lem:dimensionofLtoE0}
$\dim_{\mathbb{C}} H^{0}(X, L^{-1}\otimes E_{0}) = 2$.
\end{lem}
\begin{proof}
Due to the property of $E_{0}$, we have the exact sequence $0 \to \mathcal{O}_{X} \to E_{0} \to \mathcal{O}_{X} \to 0$. By tensoring with $L^{-1}$, we obtain
\[ 0 \to L^{-1} \to L^{-1} \otimes E_{0} \to L^{-1} \to 0. \]
 This leads to the induced long exact sequence of cohomology groups
\[ \xymatrix{
0\ar[r] & H^{0}(X, L^{-1})\ar[r] \ar@{=}[d] & H^{0}(X, L^{-1} \otimes E_{0})\ar[r] & H^{0}(X,L^{-1})\ar[r] \ar@{=}[d] & H^{1}(X, L^{-1})\ar@{=}[d] \\
 & \mathbb{C} & & \mathbb{C} & 0
}\]
\end{proof}

\begin{lem}
\label{lem:sectionwithzero}
The set $Z = \left\{ s \in H^{0}(X, L^{-1} \otimes E_{0}) \mid \exists \, p \in X \text{ such that } s(p) = 0 \right\}$ in $H^{0}(X, L^{-1} \otimes E_{0})$ forms a complex line.
\end{lem}
\begin{proof}
The inclusion $H^{0}(X, L^{-1}) \xrightarrow{i} H^{0}(X, L^{-1}\otimes E_{0})$ implies that $Z \neq \{ 0 \}$. 
 Given $0 \neq s \in Z$, let $[s]$ represent the line subbundle of $L^{-1}\otimes E_{0}$ generated by $s$ \cite[p.419]{Atiyah:1957elliptic}. Consequently, $\deg [s] = \deg \operatorname{div} s \geq 1$, and $L \otimes [s]$ is a line subbundle of $E_{0}$.

Using Lemma \ref{lem:subbundleofE0} and Remark \ref{rem:subbundleofE0}, we know that $\deg [s] = 1$ and $L \otimes [s] = \mathcal{O}_{X}$. Therefore, $s$ has precisely one simple zero denoted by $p \in X$, implying $L \cong \mathcal{O}_{X}(-p)$.

Let $0 \neq \sigma \in Z$, with $q$ being the simple zero of $\sigma$. Then $L \cong \mathcal{O}_{X}(-q)$.  Since $X$ is an elliptic curve, we have $p = q$. Furthermore, since both $s$ and $\sigma$ are sections of $L^{-1}$, they must be linearly dependent, as $\dim_{\mathbb{C}} H^{0}(X, L^{-1}) = 1$.
\end{proof}
As a consequence of Lemmas \ref{lem:dimensionofLtoE0} and \ref{lem:sectionwithzero}, we deduce the following corollary:
\begin{cor}
There exist two nowhere vanishing sections $s_{1}, \, s_{2}$ of $L^{-1} \otimes E_{0}$,  forming a basis for $H^{0}(X, L^{-1} \otimes E_{0})$.
\end{cor}
Additionally, we have the following lemma:
\begin{lem}
\label{lem:endomorphism}
Suppose $\phi \in \operatorname{End}(E_{0}) \cong \operatorname{End}(L^{-1} \otimes E_{0})$ such that:
\begin{itemize}
\item $s_{1}$ and $\phi(s_{1})$ are linearly dependent in $H^{0}(X, L^{-1} \otimes E_{0})$;
\item $s_{2}$ and $\phi(s_{2})$ are linearly dependent in $H^{0}(X, L^{-1} \otimes E_{0})$.
\end{itemize}
Then, it follows that $\phi = c \cdot \operatorname{Id}_{E_{0}}$, where $c$ is an element of $\mathbb{C}$.
\end{lem}
\begin{proof}
Note that
\begin{align*}
F &= \left\{ x \in X \mid s_{1}(x), s_{2}(x) \text{ are linear dependent as vectors in the fiber } (L^{-1}\otimes E_{0})|_{x} \right\} \\
&= \left\{x \in X \mid (s_{1} \wedge s_{2})(x) = 0 \text{ as a holomorphic section of the line bundle } \det (L^{-1} \otimes E_{0})\right\}.
\end{align*}
Since $s_{1}, s_{2}$ form a basis of $H^{0}(X, L^{-1} \otimes E_{0})$, $s_{1} \wedge s_{2}$ is not a zero section of $\det (L^{-1} \otimes E_{0})$ . Hence $F$ is a finite set of $X$.

Suppose $\phi(s_{1}) = c_{1} \cdot s_{1}$ and $\phi(s_{2}) = c_{2} \cdot s_{2}$. Then the restriction of $\phi$ to the fiber $E_{0}|_{x}$ for $(x \notin F)$ is the matrix $\operatorname{diag}(c_{1}, c_{2})$. Thus, by the holomorphic property of $\phi$, we have $\phi = \operatorname{diag}(c_{1}, c_{2})$.

Furthermore, if $c_{1} \neq c_{2}$, then $E_{0}$ decomposes into a direct sum of two line subbundles $L_{i} = \ker (\phi - c_{i} \cdot \operatorname{Id}_{E_{0}})$.
\end{proof}

\begin{proof}[Proof of Proposition \ref{prop:extensionsoflinebundle-1}]
Let $ 0 \to L \to E \to L^{-1} \to 0$ be a non-trivial extension in $\operatorname{Ext}_{X}^{1}(L^{-1}, L)$.
\begin{enumerate}
\item If $E$ is decomposable, then $E \cong J \oplus J^{-1}$, where $J \in \operatorname{Pic}^{\leq 0}(X)$. Note that if $\deg J < -1$, then there is only a zero map from $L$ to $J$  and any section in $H^{0}\big(X, L^{-1} \otimes (J\oplus J^{-1})\big)$ cannot be nowhere vanishing. Hence, $\deg J = -1$ or $0$. If $\deg J = -1$, it's evident that $J \cong L$ and $E \cong L \oplus L^{-1}$, i.e. the extension is trivial. In summary, we obtain that $\deg J = 0$ and $J$ is not a $2$-torsion by Theorem \ref{thm:inclusionindirectsum}.
\item If $E$ is indecomposable, then $E = E_{0} \otimes T$ by \cite[Theorem 5]{Atiyah:1957elliptic}, where $T$ is a $2$-torsion line bundle. Note that if $T \in \operatorname{Pic}^{0}(X)$ is nontrivial, we have $\operatorname{Hom}(E_{0}, E_{0} \otimes T) = 0$ according to \cite[Lemma 18]{Atiyah:1957elliptic}. By replacing $L$ with $L \otimes T$, we may assume $E = E_{0}$.  Then we only need to show that the equivalence classes of the extension with form
\[ 0 \to L \to E_{0} \to L^{-1} \to 0 \]
form a complex line in $\operatorname{Ext}_{X}^{1}(L^{-1}, L) \cong \mathbb{C}^{2}$.

By \cite[Lemma 17]{Atiyah:1957elliptic}, we know that $\dim_{\mathbb{C}} \operatorname{End}(E_{0}) = 2$. If $\phi \in \operatorname{Aut}(E_{0})$ and $s \in Z$, then $\phi(s) \in Z$. Hence the induced action of $\operatorname{Aut}(E_{0})$ on $\mathbb{P}H^{0}(X, L^{-1}\otimes E_{0}) \cong \mathbb{P}^{1}$ has exactly one fixed point $[Z]$. That is, The actions of $\operatorname{Aut}(E_{0})$ on $\mathbb{P}^{1}$ are affine transformations.

On the other hand, the action of $\operatorname{Aut}(E_{0}) / \{ c \cdot \operatorname{Id}_{E_{0}} \mid c \in \mathbb{C}^{*} \}$ is faithful on $\mathbb{P}^{1}$  by Lemma \ref{lem:endomorphism}. According the classification of 1-dimensional subgroups of $\operatorname{PSL}(2, \mathbb{C})$, we know that $\operatorname{Aut}(E_{0}) / \{ c \cdot \operatorname{Id}_{E_{0}} \mid c \in \mathbb{C}^{*} \} \cong \{ z \mapsto a z \mid a \in \mathbb{C}^{*}\}$ or $\{ z \mapsto z + b \mid b \in \mathbb{C} \}$.  If $\operatorname{Aut}(E_{0}) / \{ c \cdot \operatorname{Id}_{E_{0}} \mid c \in \mathbb{C}^{*} \} \cong \{ z \mapsto a z \mid a \in \mathbb{C}^{*}\}$, then the induced action of $\operatorname{Aut}(E_{0})$ on $\mathbb{P}H^{0}(X, L^{-1}\otimes E_{0}) \cong \mathbb{P}^{1}$ has two fixed points. The decomposability of $E_{0}$ is established by the proof of Lemma \ref{lem:endomorphism}. Consequently, we deduce that $\operatorname{Aut}(E_{0}) / \{ c \cdot \operatorname{Id}_{E_{0}} \mid c \in \mathbb{C}^{*} \} \cong \{ z \mapsto z + b \mid b \in \mathbb{C} \}$, implying that the action is transitive. That is, for any nowhere vanishing section $\sigma \in H^{0}(X, L^{-1}\otimes E_{0})$, there is an automorphism $\phi \in \operatorname{Aut}(E_{0})$ such that $\phi(s_{1}) = \sigma$. Hence, any extension with form $0 \to L \to E_{0}\to L^{-1} \to 0$ is equivalent to the extension
\begin{equation}
\label{equ:extensionwiths1}
0 \to L \xrightarrow{s_{1}} E_{0} \to L^{-1} \to 0,
\end{equation}
with the fixed inclusion $L \xrightarrow{s_{1}} E_{0}$. All the equivalence classes with form \eqref{equ:extensionwiths1} form a complex line in $\operatorname{Ext}_{X}^{1}(L^{-1}, L)$. Furthermore, there are four $2$-torsion line bundles on an elliptic curve. This completes our proof.
\end{enumerate}
\end{proof}

\subsubsection{Completion of the proof of Theorem \ref{thm:ell}}

Let $X$ be an elliptic curve and $D$ an effective divisor with
$\deg\, D= 2$ on $X$. Notably, the canonical line bundle $K_{X}$ of $X$ is trivial.
Since $ {\mathcal O}_X(D)$ has an even degree, we can choose a line bundle $L$ with $\deg L=-1$ such that $L^{-2}= {\mathcal O}_X(D)$. According to Theorem \ref{thm:inclusionindirectsum}, we obtain a strictly polystable short exact sequence: 
\[
0 \to L \to J \oplus J^{-1} \to L^{-1} \to 0,
\]
where $J^2\neq\mathcal{O}_{X}$. Utilizing Lemma \ref{lemma:divisor},
we can employ a Hermitian-Einstein metric on $J \oplus J^{-1}$ to derive the second fundamental form $\beta^{*} \in H^{0}(X, L^{-2})$. If $D'$ denotes the effective divisor represented by the reducible metric for the preceding short exact sequence, then as per Lemma \ref{lemma:divisor}, the associated divisor $\operatorname{Div}(\beta^{*})=D'\in |D|$.

In accordance with Proposition \ref{prop:rdm}, we establish the ramification divisor map as follows:
\begin{align*}
{\mathfrak R}_{(L, L^{-1})}\colon \mathbb{P}\Big({\rm Ext}_{X}^{1}(L^{-1},L)^{\rm sp}\Big) &\to \mathbb{P}\Big(H_{\bar{\partial}}^{1,0}\big(X, {\rm Hom}(L, L^{-1})\big)\Big) \\
[\beta] &\mapsto  [\beta^{*}]
\end{align*}
It's worth noting that we have a bijection
\[
\mathbb{P}\Big(H_{\bar{\partial}}^{1,0}\big(X, {\rm Hom}(L, L^{-1})\big)\Big)=\mathbb{P}\Big(H^{0}\big(X, {\rm Hom}(L, L^{-1})\big)\Big)\rightarrow |D|,
\]
thus yielding the map
\begin{align*}
{\mathfrak R}\colon \mathbb{P}\Big({\rm Ext}_{X}^{1}(L^{-1},L)^{\rm sp}\Big) &\to |D| \\
[\beta] &\mapsto D'.
\end{align*}
If $D'=(p)+(q)$ where $p,q\in X$ and $p\neq q$, then we have a spherical metric ${\rm d}s^2$ representing the divisor $D'$.  Now we pick a nontrivial 2-torsion point $t$ of $X$, i.e. $2t=O$, where $O$ is the identity of the group structure on $X$. Additionally, since $(p+t)+(q+t)\sim (p)+(q)$, we have $(p+t)+(q+t)\in |D|$.   Let  $f$ be a developing map of ${\rm d}s^2$, then  $\mathrm{d} s^2=f^{*}\mathrm{d} s_{0}^2$, where\ $\mathrm{d} s_{0}^2=\frac{4|{\rm d}z|^2}{(1+|z|^2)^2}$\ is the Fubini-Study metric on $\mathbb{P}^1$. Denote 
\begin{align*}
L_{t}\colon X &\to X \\
x &\mapsto x+t
\end{align*}
as the translation map, which is an isomorphic map. Since $L_{t}(p+t)=p$ and $L_{t}(q+t)=q$, $(f\circ L_{t})^{*}\mathrm{d} s_{0}^2$ is a spherical metric representing the divisor $(p+t)+(q+t)$. Since the 2-torsion subgroup of $X$ is isomorphic to $\mathbb{Z}_{2}\times\mathbb{Z}_{2}$, we conclude that there are at least 4 divisors in $|D|$, represented by cone spherical metrics.

In the case where $p=q$, meaning $D'=2(p)\in |D|$, we have $2(p+t)\sim 2(p)$. Following the previous discussion, we again find at least 4 divisors in $|D|$, represented by cone spherical metrics.

Since ${\mathfrak R}$ is smooth and $\mathbb{P}\Big({\rm Ext}_{X}^{1}(L^{-1},L)^{\rm sp}\Big)$
is path-connected, as established in Proposition \ref{prop:extensionsoflinebundle-1}, the image of ${\mathfrak R}$ is path-connected. The map ${\mathfrak R}$ is not constant and max rank ${\rm d} {\mathfrak R}\geq 1$. This implies the existence of an arcwise connected Borel subset of Hausdorff dimension one in $|D|$, containing $D'$. Moreover, every divisor in this Borel set can be represented by a family of nontrivial reducible metrics defined by one real parameter.

\section{Discussion}

In this section, we introduce two open questions with the potential to enhance the results presented in this manuscript.
\subsection{Discussion on Question \ref{ques:sp}}
By considering the Riemann theta-divisor, we formulate the following conjecture for Riemann surfaces of high genus.
\begin{conj}
\label{conj:general}
Let $X$ be a compact Riemann surface with genus $g_X \geq 2$ and $L$ a {\it generic} line bundle with $\deg L \leq -\frac{1}{2} g_{X}$. Then there exists a strictly polystable extension of $L^{-1}$ by $L$. 
\end{conj}
For the existence of embedding, we establish the following criteria:
\begin{lem}
\label{lem:baseloci}
Let $X$ be a compact Riemann surface with genus $g_X\geq 2$ and $L,\, J$ be line bundles on $X$ such that
$\deg\, L<0$ and $\deg\, J=0$. Then there exists an embedding $L \hookrightarrow J\oplus J^{-1}$ if and only if $|L^{-1}\otimes J|\not=\emptyset $, $|L^{-1}\otimes J^{-1}|\not=\emptyset$ and the base locus of $|L^{-1}\otimes J|$ does not intersect that of  $|L^{-1}\otimes J^{-1}|$.
\end{lem}
\begin{proof}
An embedding $L \hookrightarrow J\oplus J^{-1}$ gives two holomorphic sections of $L^{-1}\otimes J$ and $L^{-1}\otimes J^{-1}$, which do not vanish simultaneously. The existence of these two sections implies $|L^{-1}\otimes J|\not=\emptyset $, $|L^{-1}\otimes J^{-1}|\not=\emptyset$ and these two complete linear systems have non-intersecting base loci.

Suppose that the base loci of $|L^{-1}\otimes J|$ and $|L^{-1}\otimes J^{-1}|$ do not intersect.
If $P\in X$ does not lie on the union of these two base loci, then the linear subspace $V_P$ consisting of
pairs $(s_1,\, s_2)$, where $s_1\in H^0(X,\, L^{-1}\otimes J)$, $s_2\in H^0(X,\, L^{-1}\otimes J^{-1})$ and both of them
 vanish at $P$,  has codimension two
in $V=H^0(X,\, L^{-1}\otimes J)\oplus H^0(X,\, L^{-1}\otimes J^{-1})$.
Otherwise, there exist at most finitely many choices of $P$. Moreover, such $P$ does not lie on the two base loci simultaneously. Then $V_P$ has codimension one in $V$. By counting dimension, all the pairs $(s_1,\, s_2)\in V$ such that
$s_1$ and $s_2$ vanish at some point of $X$,  form
a proper affine subvariety of $V$. Hence, we could find such a pair $(s_1,\, s_2)\in V$ that
$s_1$ and $s_2$ does not vanish simultaneously, and $(s_1,\, s_2)$ gives an embedding
of $L$ into $J\oplus J^{-1}$.
\end{proof}

\begin{ques}
\label{ques:emb}
Let $X$ be a compact Riemann surface with genus $g_X\geq 2$ and $L$ a line bundle with $\deg\, L\leq -g_X/2$.
Does there exist $J\in {\rm Pic}^0(X)$ such that $|L^{-1}\otimes J|\not=\emptyset $, $|L^{-1}\otimes J^{-1}|\not=\emptyset$ and the base locus of $|L^{-1}\otimes J|$ does not intersect that of  $|L^{-1}\otimes J^{-1}|$?
\end{ques}
We have the following four comments about this question.
\begin{enumerate}
\item {\it If $\deg\, L\leq 2-2g_X$, there exists $J\in {\rm Pic}^0(X)$ such that $|L^{-1}\otimes J|\not=\emptyset $, $|L^{-1}\otimes J^{-1}|\not=\emptyset$ and $|L^{-1}\otimes J|\not=\emptyset $ is base free.}
Indeed,  if $\deg\, L\leq -2g_X$, then both $|L^{-1}\otimes J|$ and $|L^{-1}\otimes J^{-1}|$ are base free for all
$J\in {\rm Pic}^0(X)$ by Riemann-Roch.  As long as $\deg\, L$ equals $2-2g_X$ or $1-2g_X$,
we could find $J\in {\rm Pic}^0(X)$ since $|K_X|$ is base free.

\item {\it If $-d:=\deg\, L\leq -g_X/2$, there exists $J\in {\rm Pic}^0(X)$ such that $|L^{-1}\otimes J|\not=\emptyset $ and $|L^{-1}\otimes J^{-1}|\not=\emptyset$.} In fact, the line bundles $J\in {\rm Pic}^0(X)$ such that $|L^{-1}\otimes J|\not=\emptyset$ form a subvariety of ${\rm Pic}^0(X)$, denoted by $V_d$,  which could be identified with the image of the map
$\mu:X^{(d)}\to {\mathcal J}(X)$ (\cite[p.349]{GH:1994}).
Then, the line bundles $J\in {\rm Pic}^0(X)$ such that $|L^{-1}\otimes J^{-1}|\not=\emptyset$ form subvariety
$-V_d$ of ${\rm Pic}^0(X)$. It suffices to show that $V_d\cap (-V_d)\not=\emptyset$. By Poincar\'{e} formula (\cite[p.350]{GH:1994}), $V_d$ is homologous to
$\frac{1}{(g_X-d)!}\Theta^{g-d}$, where $\Theta$ is the Riemann theta-divisor of ${\mathcal J}(X)$. Hence,
$V_d$ has dimension $d$. Since $2d\geq g_X$, $V_d\cap (-V_d)\not=\emptyset$.

\item {\it For generic line bundles $L$ with $0>-d:=\deg\, L>-g_X/2$, there does not exist $J\in {\rm Pic}^0(X)$ such that $|L^{-1}\otimes J|\not=\emptyset $ and $|L^{-1}\otimes J^{-1}|\not=\emptyset$.} Suppose that there exists such $J$.
Then $|L^{-2}|$ consists of effective divisors of degree $2d$, which is contained in a subvariety with dimension $2d<g_X$ in ${\rm Pic}^{2d}(X)$  by Poincar\'{e} formula (\cite[p.350]{GH:1994}). This is also why we impose the condition $\deg L \leq -\frac{1}{2}g_{X}$ on the line bundles in our conjecture \ref{conj:general}.

\item \textit{There exists a line bundle $L$ of degree $-1$ on $X$ of genus $2$ (i.e., $\deg L = -\frac{1}{2}g_{X}$) such that $L$ cannot be a line subbundle of $J \oplus J^{-1}$}. It is known that $\left| K_X \right|$ defines a projective line $E$ in $\operatorname{Sym}^2(X)$, and the map $j \colon \operatorname{Sym}^2(X) \rightarrow \operatorname{Jac}(X)$ collapses $E$ to $K_X \in \operatorname{Jac}(X)$ but is otherwise bijective. Consider a point $p \in X$ such that $K_X \neq \mathcal{O}_X(2 p)$, and take $L=\mathcal{O}_X(-p)$. If $L$ were a line subbundle of $J \oplus J^{-1}$, then, according to Lemma \ref{lem:baseloci}, there would exist distinct points $q_1$ and $q_2$ in $X$ such that
\[ L^{-1} \otimes J = \mathcal{O}_X(q_1), \quad L^{-1} \otimes J^{-1}=\mathcal{O}_X (q_2). \]
This contradicts the properties of the map $j$. Based on this example, we stipulate in the conjecture \ref{conj:general} that the line bundle $L$ should be generic.

\end{enumerate}

\subsection{Further discussions on the restricted ramification divisor map}
In Proposition \ref{prop:restricted-rdm}, we introduce the concept of the restricted ramification divisor map and provide a calculation for its complex derivative.  In Lemma \ref{lemma:lower bound}, we establish a lower bound for the rank of this complex derivative. Using this result, we derive Theorem \ref{thm:gen} as a corollary, which yields a Borel subset in a complete linear system. We demonstrate that divisors in this Borel set can be represented by reducible metrics and determine the dimension of this set using the aforementioned lower bound.

Hereafter, we present a conjecture regarding the rank of the complex derivative of the restricted ramification divisor map. If this conjecture holds true, it would allow us to identify a subset with a higher Hausdorff dimension. This subset would consist of divisors that can be represented by reducible metrics.

In light of the conditions specified in Lemma \ref{lemma:lower bound}, we put forth the following.
\begin{conj}
\label{conj:rk of tangent map}
Assume that $0\neq i_{1}\in H^0\big({\rm Hom}(L,\,J)\big)$ and $0\neq i_{2}\in H^0\big({\rm Hom}(L,\,J^{-1})\big)$, then the rank of the complex derivative of the restricted ramification divisor map $\rho$ at $(i_{1},i_{2})$ is as follows:
\begin{itemize}
\item If $J^{2}\neq{\mathcal O}_X$, then $\operatorname{rank} \rho_{*(i_{1},i_{2})}= 2d-1$. In particular, $\operatorname{Ker}\rho_{*(i_{1},i_{2})}= \mathbb{C}(i_{1},-i_{2})$.

\item If $J^{2}={\mathcal O}_X$ and $i_{1}\neq \lambda i_{2}$ for any $\lambda\in\mathbb{C}$, then $\operatorname{rank} \rho_{*(i_{1},i_{2})}= 2d-3$. In particular, $\operatorname{Ker}\rho_{*(i_{1},i_{2})}= \mathbb{C}(i_{1},-i_{2})\oplus \mathbb{C}(i_{2},0)\oplus \mathbb{C}(0,i_{1})$.

\end{itemize}
\end{conj}

\noindent{\bf Acknowledgements} 
Y.F. is supported in part by the National Key R\&D Program of China No. 2022YFA1006600, Fundamental Research Funds for the Central Universities and Nankai Zhide Foundation.
J.S. is partially supported by the National Natural Science Foundation of China (Grant Nos. 12001399, 11831013 and 12171352) and the International Postdoctoral Exchange Fellowship Program by the Office of China Postdoctoral Council (NO. PC2021053). 
B.X. is supported in part by the Project of Stable Support for Youth Team in Basic Research Field, CAS (Grant No. YSBR-001) and NSFC (Grant Nos. 12271495, 11971450 and 12071449).

\end{document}